\documentclass[a4paper]{amsart}
\usepackage{tensor,braket,color, appendix}

\makeatletter
\let\@wraptoccontribs\wraptoccontribs
\makeatother

\title[Variation of total Q-prime curvature]%
{Variation of total Q-prime curvature \\ 
on CR manifolds}

\contrib[With an Appendix by]{
 A. Rod Gover and Kengo	Hirachi}

\author{Kengo Hirachi}
\address{K.H: Graduate School of Mathematical Sciences, The University of Tokyo,
	3-8-1 Komaba, Meguro, Tokyo 153-8914, Japan}
\email{hirachi@ms.u-tokyo.ac.jp}

\author{Taiji Marugame}
\address{T.M: Graduate School of Mathematical Sciences, The University of Tokyo,
	3-8-1 Komaba, Meguro, Tokyo 153-8914, Japan}
\email{marugame@ms.u-tokyo.ac.jp}

\author{Yoshihiko Matsumoto}
\address{Y.M: Department of Mathematics, Graduate School of Science and Engineering, Tokyo Institute of Technology,
	2-12-1 Ookayama, Meguro, Tokyo 152-8551, Japan}
\curraddr{Department of Mathematics, Graduate School of Science, Osaka University,
	1-1 Machikaneyama-cho, Toyonaka, Osaka 560-0043, Japan}
\email{matsumoto@math.sci.osaka-u.ac.jp}

\address{A.R.G: The University of Auckland\\
Private Bag 92019\\
Auckland 1142, New Zealand}
\email{r.gover@auckland.ac.nz}



\numberwithin{equation}{section}

\newtheorem{lem}{Lemma}[section]
\newtheorem{thm}[lem]{Theorem}

\newtheorem{prop}[lem]{Proposition}

\theoremstyle{definition}
\newtheorem{defn}[lem]{\it Definition}
\newtheorem{rem}[lem]{\it Remark}

\newcommand{\calE}{\mathcal{E}}

\newcommand{\calH}{\mathcal{H}}

\newcommand{\calJ}{\mathcal{J}}

\newcommand{\calL}{\mathcal{L}}

\newcommand{\calN}{\mathcal{N}}
\newcommand{\calO}{\mathcal{O}}
\newcommand{\calP}{\mathcal{P}}

\newcommand{\calPE}{\mathcal{PE}}

\newcommand{\bC}{\mathbb{C}}

\newcommand{\bR}{\mathbb{R}}

\newcommand{\bZ}{\mathbb{Z}}

\newcommand{\bfth}{\boldsymbol{\theta}}
\newcommand{\bfr}{{\boldsymbol{r}}}
\newcommand{\bfrho}{{\boldsymbol{\rho}}}

\DeclareMathOperator{\Ric}{Ric}

\DeclareMathOperator{\tr}{tr}
\DeclareMathOperator{\tf}{tf}
\DeclareMathOperator{\lp}{lp}
\DeclareMathOperator{\Scal}{Scal}

\DeclareMathOperator{\re}{Re}
\DeclareMathOperator{\im}{Im}

\newcommand{\wt}{\widetilde}
\newcommand{\wf}{\wt f}

\newcommand{\wE}{\wt\calE}

\newcommand{\wn}{\wt\nabla}
\newcommand{\wDelta}{\wt\Delta}

\newcommand{\conj}{\overline}

\newcommand{\pa}{\partial}
\newcommand{\wh}{\widehat}

\newcommand{\ab}{{\alpha\conj\beta}}
\newcommand{\CR}{\lrcorner\,}
\newcommand{\vol}{dv}
\newcommand{\uY}{\underline{Y}}
\newcommand{\uF}{\underline{F}}

\newcommand{\calpha}{{\conj\alpha}}
\newcommand{\cbeta}{{\conj\beta}}
\newcommand{\cA}{{\conj A}}
\newcommand{\cB}{{\conj B}}

\newcommand{\abs}[1]{\lvert#1\rvert}

\newcommand{\cl}[1]{\overline{#1}}
\renewcommand\a{\alpha}
\renewcommand\b{\beta}
\newcommand\g{\gamma}
\renewcommand\d{\delta}
\newcommand\e{\epsilon}
\newcommand\Th{\Theta}
\newcommand\ol{\overline}

\newcommand{\intf}{I_1}
\newcommand{\ints}{I_2}

\begin{document}

\begin{abstract}
We derive variational formulas for the total $Q$-prime curvature under the deformation of strictly pseudoconvex domains in a complex manifold.  We also show that the total $Q$-prime curvature agrees with the renormalized volume of such domains with respect to the complete K\"ahler--Einstein metric.
\end{abstract}
\maketitle

\section{Introduction}
There are strong analogies between CR and conformal geometries.
The theory of ambient metric by Fefferman--Graham \cite{FG3} works equally in both cases and, for example, the invariant differential operators and the $Q$-curvature are defined in a unified way.

However, there are also significant differences which stem from the embeddability (or integrability) of CR manifolds.  (The role of integrability  in connection with the deformation complex is discussed in  \cite{H3}.)  For a CR manifold $M$ embedded as the boundary of a strictly pseudoconvex domain in a Stein manifold, we can define a natural class of contact forms $\theta$, called pseudo-Einstein contact forms.  Such contact forms can be parametrized by the space of CR pluriharmonic functions $\calP$:  if $\theta$ is pseudo-Einstein, then the family of all pseudo-Einstein contact forms is given by $\calP\calE=\{e^u\theta:u\in\calP\}$.  Thus we may set up a conformal geometry of Levi metrics for which the scale factors are restricted to CR pluriharmonic functions.

It naturally follows from the definition that the $Q$-curvature vanishes for pseudo-Einstein contact forms.  Instead, we can define the $Q$-prime curvature as a secondary invariant, which shares the same properties as the $Q$-curvature in conformal geometry; see \cite{BFM,CaY,H2}.
In this paper,  we prove two of such formulas: 
\begin{enumerate}
\item 
The variational formula of the total $Q$-prime curvature in terms of the obstruction function for the Monge--Amp\`ere equation;

\item An expression for the renormalized volume of strictly pseudoconvex domains with complete K\"ahler--Einstein metric in terms of the total $Q$-prime curvature.
\end{enumerate}
We should emphasize the fact that (1) is an analogy of the variational formula in even dimensional conformal geometry \cite{GrH}, while (2) corresponds to the formula for odd dimensional conformal manifolds which bound conformally compact Einstein manifolds of even dimensions \cite{FG2}.  We will explain the correspondence in more detail in \S\ref{conf-CR-corr}.

The results of this paper have been announced in \cite{H3} together with an application of the variational formula to a rigidity theorem in CR geometry; see \S\ref{section-Variational}.  The proof of the  rigidity theorem will be given in a forthcoming paper \cite{HMO}.

\subsection{The ambient metric and $Q$-prime curvature}
To state the results, let us quickly recall the definition of the $Q$-prime curvature.  
We first consider a bounded strictly pseudoconvex domain $\Omega\subset\bC^{n+1}$
with $C^\infty$ boundary $M=\pa\Omega$.
We can fix its defining function $u$ (positive in $\Omega$)
by imposing the complex Monge--Amp\`ere equation, for which the unique existence of the solution
has been proved by Cheng--Yau \cite{ChengYau}:
$$
\calJ[u]=1\qquad \text{in} \quad \Omega,
$$
where 
$$
\calJ[u]=(-1)^{n+1}\det\begin{pmatrix}u & \pa_a u\\ \pa_{\conj b}\,u & \pa_{a\conj b}\,u
\end{pmatrix}_{a,b=1,\dots,n+1}.
$$
It follows that $\Omega$ admits an K\"ahler--Einstein metric $g_+$ with potential $-\log u$:
$$
\Ric[g_+]=-(n+2)g_+.
$$
However, in general, the exact solution has a weak singularity at the boundary; see \eqref{uexpansion}.
We thus use the best smooth approximate solution $r$ constructed by Fefferman \cite{F1}:
a defining function $r\in C^{\infty}(\bC^{n+1})$ of $\Omega$ is called a {\em Fefferman defining function} if
\begin{equation}\label{JeqC}
\calJ[r]=1+\calO\,r^{n+2}
\end{equation}
holds for an $\calO\in C^{\infty}(\bC^{n+1})$, which we call the {\em obstruction function}.
Here  $r$  modulo $O(r^{n+3})$ and $\calO$ modulo $O(r)$ are uniquely determined.
The obstruction function also appears as the logarithmic singularity of the Cheng--Yau solution $u$:
one can write
\begin{equation}
	\label{uexpansion}
	u=r(1+\eta_0 r^{n+2}+\eta_1 r^{n+2}\log r),\qquad
	\eta_0\in C^\infty(\overline{\Omega}),\quad
	\eta_1\in C^{n+1}(\overline{\Omega})
\end{equation}
and $\eta_1$ equals $(n+3)^{-1}\calO$ modulo $O(r)$; see \cite{LM}.
So if $u$ is smooth up to the boundary, then $\calO=O(r)$.
It is shown that the converse is also true; see \cite{Gr2}.

From $r$, we can define a Lorentz--K\"ahler metric $\wt g=\wt g[\bfr]$ on $\bC^*\times \bC^{n+1}$ near
$\bC^*\times M$, called the {\em ambient metric} associated to $M$,
by the K\"ahler form $i\pa\overline\pa \bfr$, where 
\begin{equation}
	\label{ambient_metric_potential}
	\bfr(\lambda,z)=|\lambda|^{{2}/({n+2})} r(z),\quad
	(\lambda,z)\in\bC^*\times \bC^{n+1}.
\end{equation}
The construction of $\bfr$ and $\wt g$ above can be generalized to a bounded strictly pseudoconvex domain $\Omega$ in a complex manifold $X$ of dimension $n+1$, where $\bC^*\times \bC^{n+1}$ is replaced by $K_X^*$, the canonical bundle of $X$ with zero-section removed.
It turns out that the locally defined function on the right-hand side of \eqref{ambient_metric_potential} on each coordinate system can be patched up to a homogeneous function $\bfr\in\wt\calE(1)$, which defines $K_X^*|_{\Omega}$ by $\bfr>0$. We call $\bfr$ the {\em Fefferman defining function} of $K_X^*|_{\Omega}$. (In general, we say that $f\in C^\infty(K^*_X)$ has {\em homogeneous degree} $(w,w)$ if it satisfies $f(\lambda v)=|\lambda|^{2w/(n+2)} f(v)$ for $\lambda\in\bC^*$ and write $f\in\wt\calE(w)$.
The space of the functions given by the restrictions of $f\in\wt\calE(w)$ over $M$ is denoted by $\calE(w)$.)
  Then the ambient metric $\wt g$ satisfies 
$$
\Ric[\wt g]=O(\bfr^{n}).
$$
The obstruction function $\calO$ can be lifted to a homogenous function in $\wt\calE(-n-2)$,
which is also denoted by $\calO$.

Let $h$ be a hermitian fiber metric of  $K_X$.
Then $h$ can be identified with a homogeneous function $h\in\wt\calE(n+2)$ and one may define a function on $X$ by the quotient $r=\bfr/h^{1/(n+2)}$.
It gives a contact form on $M$ by
\begin{equation}
\theta=d^cr\big|_{TM}, \quad\text{where }d^c=\frac{i}2(\pa-\conj\pa).
\end{equation}
We say that $\theta$ is {\em pseudo-Einstein} if it is given by a metric $h$ that is flat on the pseudoconvex side of $M$, i.e. 
\begin{equation}
\label{flat-h}
\pa\conj\pa\log h=0 \quad\text{on } 
\{0\le\bfr\ll1\}.
\end{equation}
In the case $n\ge2$, this is equivalent to Lee's pseudo-Einstein condition \cite{Lee}; see Proposition \ref{psudo-Einstein-prop} below.  It is known that $M\subset X$ admits a pseudo-Einstein contact form if $X$ is Stein \cite{CaoChang}.
 
For a pseudo-Einstein contact form, by using the associated flat metric $h_\theta$, we define 
the {\em $Q$-prime curvature} by
$$
Q'_{\theta}=\frac{1}{(n+2)^2}\wt\Delta^{n+1}(\log h_\theta)^2\big|_{\bfr=0},
$$
where $\wt\Delta$ is the Laplacian of the ambient metric $\wt g$. (Recall \cite{FH} that the $Q$-curvature for $\theta$ is defined by $-(n+2)^{-1}\wt\Delta^{n+1}\log h_\theta$, which vanishes by the condition \eqref{flat-h}).
It turns out that $Q'_\theta\in\calE(-n-1)$ and thus $Q'_\theta$ can be identified with a volume density on $M$.
The  {\em total $Q$-prime curvature} is defined  by the integral
$$
\conj{Q}'=\int_M Q'_\theta.
$$
If $\Omega$ admits a complete K\"ahler--Einstein metric $g_+$, then $\conj{Q}'$ is independent of the choice of a pseudo-Einstein contact form $\theta$ and gives a CR invariant of $M$; see  \cite{H2}.  

\subsection{Total $Q$-prime curvature and renormalized volume}\label{renormalzed-vol-sec}
We recall from \cite{H2} a characterization of the total $Q$-prime curvature in terms of a volume expansion. Suppose that $M$ admits a pseudo-Einstein contact form $\theta$ and set $r=\bfr/h_\theta^{1/(n+2)}$. Then, for the metric $g$ near the boundary with  K\"ahler form  $\omega=dd^c\log r$, 
we have, as $\epsilon\to +0$,
\begin{equation}
	\label{eq:Qprime_expansion-intro}
	\int_{\Omega^\epsilon}\abs{d\log r}^2_{g}dv_{g}
	\!=\sum_{j=0}^n\frac{ a_j}{\epsilon^{n-j+1}}
	+\frac{(-1)^n}{(n!)^3}\overline Q'\log{\epsilon}+O(1),
	\end{equation}
where 
$\Omega^\epsilon=\{z\in \Omega:r(z)>\epsilon\}$ and $dv_g=\omega^{n+1}/(n+1)!$.
If we further assume that $\Omega$ supports a complete K\"ahler--Einstein metric $g_+$, we can also compute the renormalized volume of $(\Omega,g_+)$ by using the expansion above.

\begin{thm}\label{renormalized-volume-thm}
Let $\Omega$ be a bounded strictly pseudoconvex domain of a complex manifold $X$ whose boundary
admits a pseudo-Einstein contact form $\theta$.
Assume that $\Omega$ is equipped with a complete K\"ahler--Einstein metric $g_+$ which is
given by $dd^c\log u$ near $\partial\Omega$, where $u$ is of the form \eqref{uexpansion}.
Then the volume of $\Omega^\epsilon$ admits an expansion, as $\epsilon\to+0$,
$$
\int_{\Omega^\epsilon}dv_{g_+}=\sum_{j=0}^n \frac{b_j}{\epsilon^{n-j+1}}+V+o(1).
$$
The constant term $V$, the {\em renormalized volume of $(\Omega, g_+)$}, satisfies
\begin{equation}\label{Vformula}
V=\frac{(-1)^{n+1}}{2 (n!)^3(n+1)}\,\overline Q'+
\int_\Omega\wt c_1{}^{n+1},
\end{equation}
where
 $\wt c_1\in H^2(\Omega,\pa\Omega;\bR)$ is a lift of  the first Chern class  $c_1(K_\Omega)\in H^2(\Omega;\bR)$. 
\end{thm}

Here the lift  $\wt c_1$ is defined by a 2-form in  $c_1(K_\Omega)$ with compact support, which exists when $\pa\Omega$ admits a pseudo-Einstein contact form; 
the integral of $\wt c_1{}^{n+1}$ is then independent of the choice of such a lift.
In particular, if $\Omega$ is a Stein manifold, the second term vanishes and $V$ is a multiple of $\conj{Q}'$.

\subsection{Variational formulas}\label{section-Variational}
We next consider the variation of $\overline Q'$ under a deformation of a domain $\Omega$ within a complex manifold $X$.
Let $\Omega_t$, $t\in(-1,1)$, be a family of strictly pseudoconvex domains
in $X$ such that $\Omega_0=\Omega$. We assume that the family is smooth in the sense that there is a $C^\infty$ function $\rho_t(z)$ of $(t,z)\in\bR\times X$ such that $\Omega_t=\{z\in X:\rho_t(z)>0\}$ and $d_z\rho_t\ne0$ on $\pa\Omega_t$.
For each $t$, we take a Fefferman defining function $\bfr_t$
of $K^*_X|_{\Omega_t}$. We may choose $\bfr_t$ so that it is smooth in $t$ and define derivatives
$$
\dot\bfr=\frac{d \bfr_t}{dt}\Big|_{t=0},\qquad
\ddot\bfr=\frac{d^2 \bfr_t}{dt^2}\Big|_{t=0}.
$$
We also assume that there is a fiber metric $h$ of $K_X$ that is flat in a two-sided neighborhood of
$M=\pa\Omega$ in $X$; hence, for small $\abs{t}$, $M_t=\pa\Omega_t$ admits a
pseudo-Einstein contact form $\theta_t=d^cr_t|_{TM_t}$, where $r_t=\bfr_t/h^{1/(n+2)}$,
and one may define the total $Q$-prime curvature $\overline Q'_t$.

\begin{thm}\label{variation_thm}
Let $\Omega_t$, $t\in(-1,1)$, be a smooth family of bounded strictly pseudoconvex domains in a complex manifold $X$ satisfying the assumptions as described above.  Then the total $Q$-prime curvature of $(M_t,\theta_t)$ satisfies
\begin{align}\label{QPvar}
\frac{d}{dt}\Big|_{t=0}\overline Q'_t&= c_n \int_M \dot\bfr\,\calO,
\\
\label{QPSvar}
\frac{d^2}{dt^2}\Big|_{t=0}\overline Q'_t
&=c_n\int_M
\Big( k_n\dot\bfr
P_{n+3}
\dot\bfr
+\big(\,\ddot\bfr-\abs{\pa\dot\bfr}_{\wt g}^2\,
\big)\calO
\Big).
\end{align}
Here $\calO$ is the obstruction function of $M=M_0$, $c_n=(-1)^{n+1}2n!(n+2)!$, $k_n=(-1)^n((n+2)!)^2$,
$$
P_{n+3}\colon\calE(1)\to\calE(-n-2)
$$
is a CR invariant differential operator whose principal part agrees with that of the power of the
sub-Laplacian $\Delta_b^{n+3}$,
and $\abs{\,\cdot\,}^2_{\wt g}$ is the squared norm for the ambient metric $\wt g$ for $M$.
\end{thm}

Note that $\dot\bfr|_{\calN}$, $\ddot\bfr|_{\calN}$, and $\pa\dot\bfr|_{\calN}$ are determined
by the family $\Omega_t$.
Moreover, the integrals in \eqref{QPvar} and \eqref{QPSvar} make sense because
$\dot\bfr$, $\ddot\bfr$, $\abs{\pa\dot\bfr}_{\wt g}^2\in\wt\calE(1)$.

Equation \eqref{QPvar} shows that the critical points of $\overline{Q}'$ are the domains for which the obstruction functions vanish on the boundaries.  In a neighborhood of such a domain, the behavior of $\overline{Q}'$ is controlled by the spectrum of the CR invariant operator $P_{n+3}$.  In particular, when $\Omega$ is the unit ball in $\bC^{n+1}$, we can compute the spectrum by using representation theory and show that $\overline{Q}'$ takes its local maximum at the ball in a {\em formal sense}, i.e., it holds for any one parameter family of deformations; see \cite{H3} for a precise statement.

One obtains \eqref{QPSvar} by taking the derivative of 
\eqref{QPvar} and the operator $P_{n+3}$ is derived from the first variation of the obstruction function. Explicit calculation of $P_{n+3}$ is given in the appendix by Rod Gover and the first author, where we write down the operator in terms of the connection of the ambient metric and compare it  with the power of the Laplacian for the ambient metric $\wt\Delta^{n+3}$.

\subsection{Relation to the Burns--Epstein invariant}\label{intro-BE}
We also study the relation between
the total $Q$-prime curvature and the CR invariant introduced by Burns--Epstein  \cite{BE1, BE2}.
We here recall the latter in the form later generalized by  the second author \cite{Marugame}.

Let $\Omega$ be a bounded strictly pseudoconvex domain of a complex manifold $X$
with a pseudo-Einstein contact form  $\theta$ on the boundary $M$.
Let $r=\bfr/h_\theta^{1/(n+2)}$ and take a hermitian metric $g$ on $\Omega$ that agrees with $dd^c\log r$ near the boundary.  Then the curvature $R^g$ of $g$
diverges at the boundary, but the Bochner tensor $B$, the totally trace-free part of $R^g$, is continuous up to the boundary.
Thus we may define the renormalized Chern form $c_{n+1}(B)$
as an invariant polynomial of $B$.
The integral of $c_{n+1}(B)$ can be decomposed into two parts:
$$
\int_\Omega c_{n+1}(B)=\chi(\Omega)+\mu(M),
$$
where $\chi(\Omega)$ is the Euler number of $\Omega$ and
$\mu(M)$ is a CR invariant of the boundary $M$, which we call the {\em Burns--Epstein invariant}.
A transgression formula for $c_{n+1}(B)$ gives an invariant polynomial $\Pi$ in the curvature $R$ and
the torsion $A$ of the Tanaka--Webster connection of $\theta$ on $TM$ such that
$$
\mu(M)=\int_{M}\Pi(R,A)\,\theta\wedge(d\theta)^{n}.
$$
When $n=1$, we have
$
\Pi=(4\pi)^{-2}(\Scal-4|A|^2),
$
where $\Scal$ is the scalar curvature of the Tanaka--Webster connection, and obtain (see \cite{H1}):
$$
\overline{Q}'(M)=-(4\pi)^2\mu(M).
$$
In the case $n=2$, we have
\begin{equation}\label{Pi5}
\Pi= \frac{-1}{(4\pi)^3}
\left(\frac{1}{27}\Scal^3-4R_{\a\ol{\b}\g\ol{\d}}A^{\a\g}A^{\ol{\b}\ol{\d}}+\frac{1}{3}|S|^2\Scal 
\right).
\end{equation}
Here  $|S|^2$ is the squared norm of the Chern curvature $S_{\alpha\ol{\beta}\gamma\ol{\delta}}$,
the totally trace-free part of the Tanaka--Webster curvature $R_{\alpha\ol{\beta}\gamma\ol{\delta}}$. 
Note that  $S_{\alpha\ol{\beta}\gamma\ol{\delta}}$ vanishes if and only if the CR structure is spherical, i.e., locally CR isomorphic to the sphere $S^5\subset\bC^3$.  

On the other hand, Case and Gover \cite{CGY} computed $Q'$ in the case $n=2$ and showed that
\begin{equation}\label{Q-prime5}
\begin{aligned}
\ol{Q}^{\prime}(M)+(4\pi)^3\mu(M)\\=-\int_M \Bigl(
\frac{1}{3}|S|^2\Scal & +|\operatorname{div} S|^{2}
\Bigr)\theta\wedge(d\theta)^2,
\end{aligned}
\end{equation}
where $(\operatorname{div} S)_{\alpha\ol\beta\gamma}= \nabla^{\ol\delta}S_{\alpha\ol\beta\gamma\ol\delta}$.  This corrects our error in \cite{H3}; see Remark \ref{H3error}.
We here give another proof of \eqref{Q-prime5} based on Theorem \ref{renormalized-volume-thm}
 under the assumption that $M$ is the boundary of a domain in a Stein manifold.

If $\Scal>0$, then the integrand on the right-hand side becomes non-negative, and we obtain
the following theorem, which in particular shows that $\overline Q'$ and $\mu$ are different CR invariants.

\begin{thm}
\label{thm-ineq} Let $\Omega$ be a strictly pseudoconvex domain in a Stein manifold of dimension $3$.
If $M=\pa\Omega$ admits a pseudo-Einstein contact form for which $\Scal>0$ almost everywhere, then$$
\overline Q'(M)\le -(4\pi)^3\mu(M)
$$
and the equality holds if and only if $M$ is spherical.
\end{thm}

\begin{rem}\label{H3error}\rm
In \cite{H3}, we have mistakenly claimed that the integrand on right-hand side of \eqref{Q-prime5} is $\frac{1}{3}|S|^2\Scal +4|\nabla A|^{2}$.  Fortunately, the change of   $4|\nabla A|^{2}$  by $|\operatorname{div} S|^{2}$ has no effect on the validity of Theorem \ref{thm-ineq}, which has been announced  in \cite{H3} as Theorem 3.2.  
\end{rem}

\subsection{Comparison with the conformal case}\label{conf-CR-corr}
Some of the results stated above have analogies in conformal geometry in even and odd dimensions.   We here review the correspondences.

Let $\overline X$ be a compact $(n+1)$-dimensional manifold with boundary $M=\pa X$
and $x$ a defining function of $M$ in $\overline X$ which is positive inside. 
Let $g_+$ be a Poincar\'e (or conformally compact Einstein) metric on $X$, i.e., 
$x^2 g_+$ is a Riemannian metric on $\overline X$  and $\Ric (g_+)=-ng_+$ on $X$.
The restriction $x^2 g_+|_{TM}$ defines a conformal class on $M$, called the conformal infinity of $g_+$.  For each scale $g$ in the conformal infinity, there exists a unique identification of a neighborhood of $M$ in $\overline X$ with $M\times (0,\epsilon)$ such that $g_+$ is in the form
$$
g_+=x^{-2}(dx^2+g_x),  \quad g_0=g,
$$
where $g_x$ is a 1-parameter family of Riemannian metrics on $M$.   The family of metrics has expansion 
$$
g_x=\sum_{l=0}^{\lceil n/2\rceil-1}g^{(2l)}x^{2l}+g^{(n)}x^{n}+
\begin{cases}
o(x^n) &\text{if $n$ odd},\\
\calO x^{n}\log x+o(x^n) &\text{if $n$ even}. 
\end{cases}
$$
Here $g^{(2l)}$, $g^{(n)}$, and $\calO$ are symmetric 2-tensors on $M$; $g^{(2l)}$, $l<n/2$, and $\calO$ are locally determined by $g=g_0$, while $g^{(n)}$ is determined globally.
Moreover, $\calO$ is shown to be  a conformal invariant of $[g]$ and called the {obstruction tensor}.
The CR case corresponds to the case $n$ even; see \cite{GrH2} for an explicit relation between the obstructions $\calO$.

The renormalized volume $V$ of  $(X,g_+)$ is defined to be the constant term in the expansion:
$$
\int_{\{x>\epsilon\}}dv_{g_+}=\sum_{l=0}^{\lceil n/2\rceil-1}c_l \epsilon^{2l-n}+
\begin{cases}
V^\text{odd}+o(1) &\text{if $n$ odd,}\\
L\log \epsilon+V^\text{even}+o(1) &\text{if $n$ even}. 
\end{cases}
$$
It is easy to check that $V^\text{odd}$ and $L$ are independent of the choice of $x$ (or  the scale $g\in[g]$); see \cite{Gr3}.
 Moreover $L$ is shown to be a conformal invariant of $(M,[g])$ and agrees (up to a universal constant multiple)
with the total $Q$-curvature,  the integral of the $Q$-curvature on $M$.
The expansion \eqref{eq:Qprime_expansion-intro} in the CR case can be seen as an analog of
the one for the even $n$; the log term coefficient is given by the integral of locally determined $Q$-prime curvature.  

In both parities the renormalized volume depends on $(X,g_+)$ (not just the boundary $(M,[g])$);   $V^\text{even}$ also depends on the choice of $x$.  
When $n$ is odd, Fefferman and Graham \cite{FG2} defined a non-local analogue of $Q$-curvature by 
using scattering theory for $(X,g_+)$, which we denote here by $Q^S$, and showed that
$$
V^\text{odd}=\text{const.}\int_M Q^S.
$$
The construction of $Q^S$ can also be applied to the case $n$ even; 
Chang, Qing and Yang \cite{CQY} showed that $Q^S$ is a conformal primitive of the $Q$-curvature and satisfies
$$
V^\text{even}=\text{const.}\int_M (Q^S+v),
$$
where $v$ is a local invariant of the scale in $[g]$ that corresponds to $x$.
In this respect, the total $Q$-prime curvature looks more like $V^\text{odd}$. 

We next consider the variation of $L$ and $V$ for a given family of Poincar\'e metrics $g^t_+$ on $X$ with conformal infinities $[g^t]$ on $M$.  Set $g=g^0$ and $\dot g=(d/dt)|_{t=0}g^t$. Then the conformal invariant $L([g^t])$ satisfies
$$
\frac{d}{dt}L([g^t])\bigg|_{t=0}
=\text{const.}\int_M \langle \calO,\dot g\rangle dv_g,
$$
where $\calO$ is the obstruction tensor for $[g^0]$ and the inner product is defined by the scale $g=g^0$; see \cite{GrH}.  It corresponds to the variational formula  in Theorem \ref{variation_thm}.
When $n$ is odd, $V^\text{odd}$ depends only on $g_+^t$ and  satisfies
$$
\frac{d}{dt}V^\text{odd}(g_+^t)\bigg|_{t=0}
=\text{const.}\int_M \langle g^{(n)}, \dot g\rangle dv_g,
$$
where 
$g^{(n)}$ is the globally determined term in the expansion of $g_x$; see \cite{Al,An}.
When $n$ is even, $V^\text{even}$ also depends on the choice of a scale $g^t$ and we have
$$
\frac{d}{dt}V^\text{even}(g_+^t,g^t)\bigg|_{t=0}
=\text{const.}\int_M \langle g^{(n)}+H(g), \dot g\rangle dv_g,
$$
where $H(g)$ is a 2-tensor locally determined by $g$; see \cite{CQY, GMS}.  In both cases, unlike $L$, the variations are given by  globally determined terms in $g_+$.  

Finally, let us emphasize the fact that the defining functions used in the renormalization are different in conformal and CR cases: $x$ is chosen so that $g_+$ becomes diagonal, and the Fefferman defining function $r$ is the exponential of a K\"ahler potential of $g_+$.   Hence our definition of the renormalized volume for strictly pseudoconvex domains are not the precise counterpart of the conformal case.  (See \cite{HPT, Matsumoto}  for studies of the volume expansion for $x$ in the CR case.) However, the defining function $r$ is the key tool in the parabolic invariant theory \cite{F2, H01} 
and has better biholomorphic invariance properties; we believe the new renormalized volume should be a fundamental object of study.

\subsection{Plan of the paper}
In \S2, we recall basic tools in the differential geometry of CR manifolds including the Tanaka--Webster connection, the ambient metric and pseudo-Einstein contact forms. We here use the flat metric of the canonical bundle to define  pseudo-Einstein contact forms, which is useful when we study the deformations of strictly pseudoconvex domains.   In \S3, we define the $Q$-prime curvature and express the renormalized volume in terms of the total $Q$-prime curvature $\overline{Q}'$.  \S4 is a preliminary to the proof of the variational formulas.  We use the Hamiltonian flow
on the ambient space to describe deformation of domains in a complex manifold. 
In \S5, we prove the first and second variational formulas of the total $Q$-prime curvature.
The last section, \S6, is devoted to an explicit calculation of $\overline{Q}'$ for $5$-dimensional CR manifolds.
In the appendix, by Rod Gover and the first author,  we study the property  of CR invariant differential operator $P_{n+3}$ that appears in the second variational formula of $\overline{Q}'$.  

\medskip
\noindent
{\em Notations.}
We use Einstein's summation convention and assume that 
\begin{itemize}
\item
uppercase Latin indices $A,B,C,\dots$ run from 0 to $n+1$;
\item
lowercase Latin indices
$a,b,c,\dots$
run from 1 to $n+1$; 
\item
lowercase Greek indices $\alpha,\beta,\gamma,\dots$ run from 1 to $n$. 
\end{itemize}
The letter $i$ denotes $\sqrt{-1}$. 

\section{Preliminaries}

\label{Preliminary}
\subsection{Pseudo-hermitian geometry}
Let $\Omega$ be a relatively compact domain with $C^\infty$ boundary in a complex manifold $X$ of dimension $n+1\ge2$.  Then the boundary $M=\pa\Omega$ inherits a {\em CR structure}
$T^{1,0}:=\bC TM\cap T^{1,0}X$. We set $T^{0,1}=\conj{T^{1,0}}$ and $\calH=\re T^{1,0}$.
Then there is a real 1-form $\theta$ on $M$ such that $\ker\theta=\calH$.
The {\em Levi form} of $\theta$ is the hermitian form on  $T^{1,0}$ defined by
$$
L_\theta(Z,W)=-2id\theta(Z,\conj{W}).
$$
We assume that $T^{1,0}$ is {\em strictly pseudoconvex} in the sense that the Levi form is positive definite for a choice of $\theta$; such a $\theta$ is called a {\em pseudo-hermitian structure}, or a (positive) {\em contact form}.
If we take a defining function $\rho\in C^\infty(X)$ of $M$ which is positive in $\Omega$, then a contact form is given by
\begin{equation}\label{normalrho}
\theta=d^c\rho|_{TM},\quad\text{where } d^c=\frac{i} {2}(\pa-\conj\pa).
\end{equation}
 In this case, we say that {\em $\rho$ is normalized by $\theta$}. 
For a choice of a contact form $\theta$, we may uniquely determine a real vector field $T$ on $M$, called the {\em Reeb vector field}, by
$$
T\CR d\theta=0,\quad \theta(T)=1.
$$
Let us take a frame $Z_\alpha$ of $T^{1,0}$ and set  $Z_{\conj\beta}=\conj{Z_\beta}$.  Then the set
$$
T,\  Z_\alpha, \ Z_{\conj\beta}
$$
forms a frame of $\bC TM=\bC T\oplus T^{1,0}\oplus T^{0,1}$.
The dual frame $\theta,\theta^\alpha,\theta^{\conj\beta}$ is said to be an {\em admissible coframe} and satisfies
$$
 d\theta=i\, \ell_\ab\theta^\alpha\wedge\theta^{\conj\beta}
$$
for a positive hermitian matrix $\ell_{\alpha\conj\beta}$.
We will use Greek indices to refer to the frames of $T^{1,0}$ and $(T^{1,0})^*$ and denote the bundles by $\calE^\alpha$ and $\calE_\alpha$, respectively; we will also use $\calE^\alpha$ and $\calE_\alpha$ to denote  the corresponding  spaces of smooth sections.  The conjugate bundles are denoted by $\calE^{\conj\alpha}$, $\calE_{\conj\alpha}$ and the tensor products of these bundles by the list of indices, e.g., $\calE_{\alpha\conj\beta}$ stands for $(T^{1,0})^*\otimes(T^{0,1})^*$.

Let $K_X=\wedge^{n+1} (T^{1,0}X)^*$ be the canonical bundle of $X$.  The restriction of $K_X$ over $M$ defines the canonical bundle $K_M$ of $M$, which can be identified with the subbundle of $\wedge^{n+1}\bC T^*M$ defined by the equation $Z\CR\zeta=0$ for all $Z\in T^{0,1}$.
Deleting the zero sections from these bundles, we define $\bC^*$-bundles 
$\wt X=K_X\setminus\{0\}$ and $\calN=K_M\setminus\{0\}$.  Then $\calN$ is a pseudoconvex real
hypersurface of $\wt X$; the Levi form degenerate in the fiber direction.
For $\lambda\in \bC^*$, we define the {\em dilation}
$\delta_\lambda\colon \wt X\to \wt X$ by scalar multiplication
$$
\delta_\lambda(\zeta)=\lambda^{n+2}\zeta.
$$ 

For $w\in\bR$, the {\em density bundle of weight $w$ over $X$} is defined by 
$$
\wt\calE(w)=(K_X)^{-w/(n+2)}\otimes(\conj{K_X})^{-w/(n+2)}
$$
and the one over $M$ is defined by $\calE(w)=\wt\calE(w)|_M$, the restriction of $\wt\calE(w)$ over $M$.
As before, $\wt\calE(w)$ and $\calE(w)$ will also stand for the space of {\em densities}, i.e.,
sections of these bundles. 
While the fractional power $(K_X)^{-1/(n+2)}$ may exist only locally, these bundles are globally defined.
The densities can be identified with functions of
{\em homogeneous degree }$(w,w)$:
$$
\wt\calE(w)=\{f\in C^\infty(\wt X,\bC): \delta_\lambda^*f=
|\lambda|^{2w}f \quad\text{for any}\ \lambda\in\bC^*\}
$$
and $\calE(w)=\{f|_\calN\in C^\infty(\calN,\bC):f\in\wt\calE(w)\}$.

 We next recall the relation between a contact form $\theta$ and a non-vanishing section $\zeta$ of $K_M$. 
 Let $T$ be the Reeb vector field of $\theta$. Then we say that $\theta$ {\em is normalized with respect to $\zeta$} if
$$
\theta\wedge d\theta^n=i^{n^2}n!\,\theta\wedge(T\CR\zeta)\wedge(T\CR\conj\zeta)
$$
holds.  If $\wh\zeta=\lambda\zeta$, then $\wh\theta$ normalized by $\wh\zeta$ is given by
$\wh\theta=|\lambda|^{2/(n+2)}\theta$.  Since $|\zeta|^{-2/(n+2)}$ defines a section of $\calE(1)$, 
the section $\bfth:=\theta\otimes |\zeta|^{-2/(n+2)}$ of $T^*M\otimes \calE(1)$ is independent of $\theta$.
Under the scaling $\wh\theta=e^u\theta$, the Levi metric is scaled by $\wh\ell_{\alpha\conj\beta}=e^u\ell_{\alpha\conj\beta}$; hence we have a canonical section 
$\boldsymbol{l}_{\alpha\conj\beta}\in\calE_{\alpha\conj\beta}(1):=\calE_{\alpha\conj\beta}\otimes\calE(1)$ and its inverse $\boldsymbol{l}^{\alpha\conj\beta}\in\calE^{\alpha\conj\beta}(-1)$. We will use $\boldsymbol{l}_{\alpha\conj\beta}$ and $\boldsymbol{l}^{\alpha\conj\beta}$ to lower and raise the indices unless otherwise stated.

A choice of $\theta$ determines 
the {\em Tanaka--Webster connection} $\nabla$ on $TM$: it is given in terms of an admissible frame by
$$
\nabla Z_\alpha=\omega_\alpha{}^\beta\otimes Z_\beta,\quad
\nabla Z_{\conj\alpha}=\omega_{\conj\alpha}{}^{\conj\beta}\otimes Z_{\conj\beta},\quad \nabla T=0.
$$
The connection forms $\omega_\alpha{}^\beta$,  $\omega_{\conj\alpha}{}^{\conj\beta}=\conj{\omega_\alpha{}^\beta}$ satisfy
$$
d\theta^\beta=\theta^\alpha\wedge\omega_\alpha{}^\beta+A^\beta{}_{\conj\gamma}\,\bfth\wedge\theta^{\conj\gamma},
\quad \omega_{\alpha\conj\beta}+\omega_{\conj{\beta}\alpha}=d\ell_{\alpha\conj\beta},
$$
where the indices of the connection forms are lowered by $\ell_{\alpha\conj\beta}$.
The tensor $A_{\alpha\beta}=\conj{A_{\conj\alpha\conj\beta}}\in\calE_{\alpha\beta}$ is shown to be symmetric and is called the {\em Tanaka--Webster torsion}.
There is an induced connection on the canonical bundle $K_M$ and also on the density bundles $\calE(w)$.
Then it is known that $\nabla\bfth=0$ and $\nabla\boldsymbol{l}=0$.
The curvature of the connection is defined by
$$
d\omega_\alpha{}^\beta-\omega_\alpha{}^\gamma\wedge\omega_\gamma{}^\beta=
R_\alpha{}^\beta{}_{\rho\conj\sigma}\theta^\rho\wedge\theta^{\conj\sigma}\mod
\theta,\theta^\rho\wedge\theta^\sigma,\theta^{\conj\rho}\wedge\theta^{\conj\sigma}.
$$
We call $R_{\alpha}{}^\beta{}_{\gamma\conj\sigma}\in\calE_{\alpha}{}^\beta{}_{\gamma\conj\sigma}$ the {\em Tanaka--Webster curvature}. 
The Ricci tensor and the scalar curvature are defined  by
$$
\Ric_{\alpha\conj\beta}=R_{\gamma}{}^{\gamma}{}_{\alpha\conj\beta},
\quad
\Scal=\Ric_{\alpha}{}^\alpha.
$$
We denote the components of successive covariant derivatives of a tensor by subscripts preceded by a comma, 
as in $A_{\alpha\beta,\gamma\conj\sigma}$.  
When the derivatives are applied to a function, we omit the comma. 
 With these notations, we set
$$
\pa_bf=f_\alpha \theta^\alpha,\quad\conj\pa_bf=f_{\conj\alpha} \theta^{\conj\alpha}.
$$
We call a function satisfying $\conj\pa_bf=0$ a {\em CR holomorphic function} and call
a real valued function that can be locally written as the real part of a CR holomorphic function a {\em CR pluriharmonic function}.  A CR holomorphic (resp.\ CR pluriharmonic) function  can be extended to a holomorphic (resp.\ pluriharmonic)  function on the pseudoconvex side of $M$. 

We will use the index $0$ to denote the $\bfth$ component. Then the commutator of the derivatives on $f\in\calE(w)$ 
are given by
\begin{equation}\label{commf}
2f_{[\alpha\beta]}=0,\qquad
2f_{[\alpha\conj\beta]}=i \,\boldsymbol{l}_{\alpha\conj\beta}f_0,\qquad
2f_{[0\alpha]}=A_{\alpha\beta} f^\beta
\end{equation}
and the Bianchi identities give
\begin{equation}\label{BianchiA}
A_{\alpha[\beta,\gamma]}=0,
\quad
A_{\alpha\beta},{}^{\alpha\beta}+A_{\conj\alpha\conj\beta},{}^{\conj\alpha\conj\beta}=\Scal_{0}.
\end{equation}
Here $[\cdots]$ indicates antisymmetrization over the enclosed indices.
We also recall that
\begin{equation}\label{tromega}
d\omega_\gamma{}^\gamma=
\Ric_{\alpha\conj\beta}\theta^\alpha\wedge\theta^{\conj\beta}
+A_{\gamma\alpha,}{}^\gamma\theta^\alpha\wedge\bfth
-A_{\conj\gamma\conj\beta,}{}^{\conj\gamma}\theta^{\conj\beta}\wedge\bfth.
\end{equation}
See \cite[\S2]{Lee} for the proofs of the formulas listed above.

We next define the {\em Graham--Lee connection} on $TX$ which is compatible with the Tanaka--Webster connection 
on each leaf $\{\rho=\e\}$.
Let 
$$
\vartheta=d^c\rho.
$$
Then there exists a uniquely determined  $(1,0)$-vector field $\xi$ on $X$ near $M$ 
such that
$$
\xi\CR d\vartheta=i\kappa\conj\pa\rho,\quad \xi \rho=1
$$
holds a  smooth function $\kappa$, called the {\em transverse curvature} of $\rho$.
We set 
\begin{equation}\label{def-xi}
\xi=N+\frac{i}{2}\,T
\end{equation}
for real vector fields $N$ and $T$. Then $N\rho=1$ and
$T$ gives the Reeb vector field on each leaf with a contact form induced by $\vartheta$. If we set 
$$
\ker\pa\rho=\{Z\in T^{1,0}X: Z\rho=0\},
$$
then 
$T^{1,0}X=\ker\pa\rho\oplus\bC\xi$.
Let $Z_\alpha$ be local sections of $\ker\pa\rho$ that form a local frame. Then setting $Z_{n+1}=\xi$, we obtain a local frame $Z_a$ of  $T^{1,0}X$.
The dual frame $\theta^a$ of $Z_a$ satisfies $\theta^{n+1}=\pa\rho$ and 
$$
d\vartheta=i\,\ell_{\alpha\conj\beta}\theta^\alpha\wedge\theta^{\conj\beta}+i\kappa\pa \rho\wedge\conj\pa\rho.
$$
\begin{prop}[\cite{GL}]\label{GLcurv}
There exists a unique linear connection $\nabla$ on $TX$ near $M=\{\rho=0\}$ with the properties:
\begin{itemize}
  \item[(a)] $\nabla$ preserves $\xi$ and  $\ker\pa\rho$;
  \item[(b)] 
  The connection forms ${\varphi _{\alpha}}^{\beta}$ on $\ker\pa\rho$,
  defined by $\nabla Z_{\alpha}={\varphi_{\alpha}}^{\beta}\otimes Z_{\beta}$, satisfy the structure equations
             \begin{gather}\label{GLstr}
             d\theta^{\alpha}=\theta^{\beta}\wedge{\varphi_{\beta}}^{\alpha}+i{A^{\alpha}}_{\ol{\beta}}\partial\rho\wedge\theta^{\ol{\beta}}-\ell^{\alpha\ol{\beta}}(Z_{\ol{\beta}}\kappa)\partial\rho\wedge\ol{\partial}\rho-\frac{1}{2}\kappa d\rho\wedge\theta^{\alpha},\\
             \varphi_{\alpha\conj\beta}+\conj{ \varphi_{\beta\conj\alpha}}=d\ell_{\alpha\conj\beta},
             \end{gather} 
 where the indices are lowered by using $\ell_{\alpha\conj\beta}$.
\end{itemize}
\end{prop}

The structure equations imply that if $Y$ and $Z$ are tangent to $M^{\e}=\{\rho=\e\}$,  then
$$
\nabla_{Y}Z={\nabla^{\e}}_{Y}Z,
$$
where $\nabla^{\e}$ is the Tanaka--Webster connection on $M^{\e}$.

\subsection{The ambient metric}\label{sec-ambient-metric}
Let $\bfrho\in\wt\calE(1)$ be a defining function of $\calN\subset\wt X$ that is positive on the pseudoconvex side.  Then
$$
\wt g[\bfrho]=(\wt g_{A\conj B})=(\pa_A\pa_{\conj B}\bfrho)
$$
 gives a Lorentz--K\"ahler metric in a neighborhood of $\calN$, which is called a {\em pre-ambient metric}.  Here the upper case Latin indices $A, B$ run through $0,1,\dots, n+1$.
 %
 %
 We will use $\wt g_{A\conj B}$ and its inverse $\wt g^{A\conj B}$ to lower and raise the upper case  indices.
 Let us take coordinates $z^a$ of $X$ and define a fiber coordinate $\lambda$ of $\wt X=K_X^*$
 by  $\lambda dz^1\wedge\cdots\wedge d z^{n+1}$.  Then choosing a branch 
 $z^0=\lambda^{1/(n+2)}$, we may write 
$
\bfrho=|z^0|^2 \rho
$ 
for a function $\rho$ on $X$.  
Let $Z_0=z^0\pa_{z^0}$, which is independent of the choice of branch.  Then combining with the frame $Z_a$ of $T^{1,0}X$ given in the definition of Graham--Lee connection, we obtain a frame
$Z_A=(Z_0,Z_a)$ of $T^{1,0}\wt X$ and its  dual frame
$$
\theta^0=d\log z^0, \quad \theta^\alpha,\quad \theta^{n+1}=\pa\rho.
$$
 With these frame/coframe, 
 the components of the pre-ambient metric and its inverse are given by
\begin{equation}\label{wtgatN}
\begin{aligned}
\wt g_{A\conj B}&=|z^0|^2
\begin{pmatrix}
\rho & 0 &1\\
0&   -\ell_{\alpha\conj\beta} &0\\
1 & 0 &- \kappa
\end{pmatrix},
\\
\wt g^{A\conj B}&=|z^0|^{-2}
\begin{pmatrix}
\kappa & 0&1\\
0 & -\ell^{\alpha\conj\beta}&0\\
1 &0&0
\end{pmatrix}+O(\rho).
\end{aligned}
\end{equation}
In view of this, we see that the dual vector field of $\ol{\pa}\bfrho$ is $Z_0$.
Therefore, for $F\in\wt\calE(w)$, we have $\bfrho^A F_A=wF$.

In the computations using $\wt g$, it is also useful to employ homogeneous coordinates $X^A$ defined by
$$
X^0=z^0,\quad X^a=z^0 z^a.
$$
Then, with the frame $\pa_A=\pa/\pa X^A$, 
we have $X^A \bfrho_A=\bfrho$. Hence applying $\pa_{\conj B}$, we get
$X^A \wt g_{A\conj B}=\bfrho_{\conj B}$ and so
$
\bfrho^A=X^A,
$ where the index is raised by $\wt g^{A\conj B}$.
Since the Christoffel symbol is given by 
$\Gamma_{AB}^C=\wt g^{C\conj D}\pa_{AB\conj D}\bfrho$, we have
\begin{equation}\label{rhoAB}
\bfrho_{AB}=\pa_{A}\bfrho_B-\Gamma_{AB}^C\bfrho_C=
\pa_{AB}\bfrho-X^{\conj C}\pa_{AB\conj C}\bfrho=0.
\end{equation}

We normalize $\bfrho$ by imposing a complex Monge--Amp\`ere equation.
On the canonical bundle $K_X$, there is a tautological $(n+1,0)$-form $\wt\zeta$
and hence $d\wt\zeta$ gives a canonical $(n+2,0)$-form on $\wt X$.
It gives a volume form on $\wt X$:
$$
\wt\vol=i^{(n+2)^2}d\wt\zeta\wedge \conj{d\wt\zeta}.
$$
We consider the Monge--Amp\`ere equation for $\bfrho\in\wt\calE(1)$ in a neighborhood of $\calN\subset\wt X$:
$$
(dd^c\bfrho)^{n+2}=k_n\,\wt\vol.
$$
Here $ k_n=-{(n+1)!}/{(n+2)}$, which is chosen so that  this equation is reduced to $\calJ[\rho]=1$ in each set of local coordinates $z$ on $X$.
Fefferman's theorem  \cite{F1} on  approximate solution to $\calJ[\rho]=1$ can be reformulated as follows:

\begin{prop}\label{MA-prop}
{\rm (i)}
There exists a defining function $\bfr\in\wt\calE(1)$ of $\calN$ such  that
\begin{equation}\label{MA-eq}
(dd^c\bfr)^{n+2}=k_n(1+\calO\bfr^{n+2})\wt\vol
\end{equation}
for a function $\calO\in\wt\calE(-n-2)$.  Moreover, such an $\bfr$ is unique modulo $O(\bfr^{n+3})$ and $\calO$ modulo $O(\bfr)$ is independent of the choice of $\bfr$. 
\end{prop}

\begin{defn}
We call $\bfr$ and $\calO$, respectively, the {\em Fefferman defining function} and
the {\em obstruction function}.
The {\em ambient metric} of $M$ is defined to be $\wt g[\bfr]$.
\end{defn}

The Fefferman defining function $\bfr$ gives a 1-form on $\calN$:
$$
\bfth=d^c\bfr|_{T\calN},
$$ 
which can be identified with the previously defined
$\bfth=\theta\otimes |\zeta|^{-2/(n+2)}$. In fact, for a section $\zeta$ of $\calN\to M$, the pull-back
$\theta=\zeta^*\bfth$ is normalized with respect to $\zeta$; see
 \cite[Prop. 5.2]{Farris}. 
In terms of $\bfth$ on $\calN$, the identification of $\varphi\in\calE(-n-1)$ with a volume form on $M$ is given as follows.  Since $\delta^*_\lambda \bfth=|\lambda|^2\bfth$, we have
$\delta^*_\lambda \bfth\wedge (d \bfth)^n=|\lambda|^{2(n+1)}\bfth\wedge (d \bfth)^n$.
It follows that $\varphi\bfth\wedge (d \bfth)^n$ is invariant under $\delta^*_\lambda$ and
the pull-back $\zeta^*(\varphi\bfth\wedge (d \bfth)^n)$ is a volume form on $M$ which is independent of  $\zeta$.

\subsection{Flat metric on the canonical bundle and pseudo-Einstein condition}
For  a hermitian CR  line bundle $L$ over $M$ (in the sense that the transition functions are chosen to be CR holomorphic), we may lift $\conj\pa_b$ to $$\conj\pa_b\colon C^\infty(M,L)\to C^\infty(M,(T^{0,1})^*\otimes L).$$
We consider a connection $D\colon C^\infty(M,L)\to C^\infty(M,\bC T^*\otimes L)$
such that the restriction
$$D^{0,1}\colon C^\infty(M,L)\to C^\infty(M,(T^{0,1})^*\otimes L)$$
agrees with $\conj\pa_b$.
By analogy with the construction of the Chern connection of holomorphic hermitian vector bundles over complex manifolds, we can uniquely define a metric connection $D$ for the direction
$T^{1,0}\oplus T^{0,1}$.  However, to fix the derivative in the transversal direction, we need to impose an additional normalization 
on the curvature 
$$
R^D(V,W)\zeta=D_V D_W\zeta-D_WD_V\zeta-D_{[V,W]}\zeta.
$$
The {\em Tanaka connection} of a CR line bundle with a hermitian metric $(L,h)$ is a metric connection uniquely determined by the normalizations 
$D^{0,1}=\conj\pa_b$ and
$$\tr R^D:=\boldsymbol{l}^{\alpha\conj\beta}R^D(Z_\alpha,Z_{\conj \beta})=0,$$
where $Z_\alpha$ is a frame of $T^{1,0}$; see  \cite{T}.
We say that the metric $h$ is {\em flat} if the Tanaka connection is flat. 
See \cite{DT} for more comprehensive study of the connections on CR bundles.

Now we specialize $L$ to the canonical bundle $K_M$.
We can see that the Tanaka connection is consistent with the Chern connection on 
the canonical bundle of $X$ as follows.
Let $\Omega\subset X$ be a strictly pseudoconvex domain and $M=\pa\Omega$.
Let $\wt h$ be a hermitian metric on $K_X$, which can be identified with a homogenous function
$\wt h\in\wt\calE(n+2)$. Hence its restriction to $\calN$ gives a density $h\in\calE(n+2)$.  We fix an ambient metric $\wt g[\bfr]$ of $M$ and denote the K\"ahler Laplacian by $\wt\Delta=\wt g^{A\conj B}\wt\nabla_{A\conj B}$. 

\begin{lem}\label{curvature-lem}
Let $D$ be the connection of $K_M$ given by the restriction of
the Chern connection $\wt D$ of $(K_X,\wt h)$. Then  $D$ is the Tanaka connection
of $(K_M,h)$
 if and only if $\wt\Delta \log \wt h=O(\bfr)$.  
 \end{lem}
\begin{proof}
It suffices to verify the  condition $\tr R^D=0$. 
The curvature form of $\wt D$ is
$
dd^c\log \wt h
$, and that of $D$ is given by the restriction.
Since $\wt h$ is locally expressed as $|z_0|^2h$, where $h$ is a basic function,
it follows that $Z_0\CR\pa\conj\pa\log\wt h=0$.
Hence by the  matrix expression of $\wt g^{A\conj B}$ given in \eqref{wtgatN}, we have
$$
\tr R^D=(\wt\Delta\log\wt h)|_\calN.
$$
Hence $\tr R^D=0$  is equivalent to $\wt\Delta \log \wt h=O(\bfr)$.
\end{proof}

We next compare the Tanaka connection with the connection induced from the Tanaka--Webster connection. For a contact form $\theta$, take a section $\zeta$ of $K_M$ which normalizes $\theta$.
Then there is a unique metric $h_\theta$ of $K_M$ such that $|\zeta|_{h_\theta}=1$.
This gives a one-to-one correspondence between the contact forms on $M$ and metrics on $K_M$.
Let $\nabla$ be the Tanaka--Webster connection on $TM$ for $\theta$. Then there is an induced metric connection on $(K_M,h_\theta)$, which is also denoted by $\nabla$.
The relation between $\nabla$ and the Tanaka connection $D$ is given as follows.

\begin{lem}\label{curvature-lem2}
For any section $\zeta$ of $K_M$, one has
\begin{equation}\label{connection-relation}
D\zeta=\nabla\zeta-\frac{i} {n}\Scal\theta\otimes\zeta,
\end{equation}
and the curvature of $D$  is given by
\begin{equation}\label{bPform}
R^D=-(\tf\Ric)_{\alpha\conj\beta}\theta^\alpha\wedge\theta^{\conj\beta}
-S_\alpha\theta^\alpha\wedge\bfth+
S_{\conj\beta}\theta^{\conj\beta}\wedge\bfth,
\end{equation}
where $(\tf\Ric)_{\alpha\conj\beta}=\Ric_{\alpha\conj\beta}-\frac1{n}\Scal\,\boldsymbol{l}_{\alpha\conj\beta}$ and
$
S_\alpha=\conj{S_{\conj\alpha}}=\frac{i}{n}\Scal_\alpha+A_{\alpha\gamma,}{}^\gamma.
$
 \end{lem}
\begin{proof}
 The curvature of $\nabla$ on $K_M$ is given by (see \cite[\S4]{Lee})
$$
R^\nabla(Z_\alpha,Z_{\conj \beta})=-\Ric_{\alpha\conj\beta}.
$$
It follows that $\tr R^\nabla=-\Scal$.  Thus the difference between $D$ and $\nabla$ is given by 
\eqref{connection-relation} and  the connection form of $D$ is
$$
\eta=-\omega_\alpha{}^\alpha-\frac{i} {n}\Scal \theta.
$$
Using \eqref{tromega}, we obtain \eqref{bPform}.
\end{proof}

Now we formulate Lee's pseudo-Einstein condition (originally given in the case $n\ge2$) in terms of the metric on the canonical bundle.

\begin{prop}\label{psudo-Einstein-prop}
For a contact form $\theta$ on $M$, the following are equivalent:
\begin{enumerate}
\item The metric
$h_\theta$ of $K_M$ is flat.
\item
The metric
$h_\theta$ extends to a metric of $K_X|_{\ol{\Omega}}$ that is flat
near the boundary.
\item
Locally, $\theta$ is normalized with respect to a closed section of $K_M$.

\item The Tanaka--Webster connection of $\theta$ satisfies
$$
\begin{cases}\tf \Ric=0 &\text{if } n>1,\\
\Scal_1-iA_{11,}{}^{1}=0& \text{if } n=1.
\end{cases}
$$
\end{enumerate}
In (2), the flat extension of $h_\theta$ is unique near the boundary.
\end{prop}

\begin{defn}
A contact form $\theta$ is {\em pseudo-Einstein} if one of the equivalent conditions in the proposition above holds.  We denote by $\calPE$ the set of all pseudo-Einstein contact forms on $M$.
\end{defn}

\begin{proof}[Proof of Proposition \ref{psudo-Einstein-prop}]
The Bianchi identity $dd\eta=0$ for the curvature of $D$ gives
$$
(\tf\Ric)_{\alpha\conj\beta,}{}^{\conj\beta}=-i(n-1)S_\alpha.
$$
Hence, in view of \eqref{bPform},  the condition (1) is equivalent to (4) in  all dimensions.

By Lemma \ref{curvature-lem}, if we extend $h_\theta$ to $\wt h$ by imposing
$\wt\Delta\log \wt h=O(\bfr)$, then the curvature of $h_\theta$ is given by the restriction of $dd^c\log \wt h$ to $M$.
Hence (1) holds if and only if $\log h_\theta$ is CR pluriharmonic; see \cite[Lem.~3.2]{H2}.
 This is equivalent to (2) asserting that the extension $\log\wt h$ is pluriharmonic on the pseudoconvex side. The equivalence of (3) and (4) has been proved in \cite[Thm.~4.2]{Lee} in the case $n\ge2$ and \cite[Prop. 5.1]{H2} in the case $n=1$.
 \end{proof}

By the condition (1) or (3), it is clear that if $\theta\in\calPE$, then $e^u\theta\in\calPE$ if and only if $u$ is CR pluriharmonic.

\section{Total $Q$-prime curvature and renormalized volume}
\subsection{Definition of the total $Q$-prime curvature}
For a pseudo-Einstein contact form $\theta\in\calPE$, we have a flat metric $h_\theta$ on $K_M$ and its unique flat extension, which we also denote by $h_\theta$. 

\begin{defn}
The {\em $Q$-prime curvature} of $\theta\in\calPE$ is defined by
\begin{equation}\label{defQp}
Q_\theta'=(n+2)^{-2}\wt\Delta^{n+1}(\log h_\theta)^2|_{\calN}\in\calE(-n-1).
\end{equation}
\end{defn}

The constant $(n+2)^{-2}$ is chosen so that the definition agrees with the one in \cite{H2},
where the bundle $K_X^{1/(n+2)}$ is used in place of $K_X$.
The {\em total $Q$-prime curvature} of $(M,\theta)$ is defined to be the integral
$$
\overline Q'=\int_M Q_\theta'.
$$
Using the flat metric $h_\theta$, we set
$$
r=\bfr/ h_\theta^{1/(n+2)},
$$
which we call the {\em Fefferman defining function} associated with $\theta$, and
 define a K\"ahler metric $g$ near the boundary by the K\"ahler form
$\omega=dd^c\log r$. 
Then  we can write $\overline Q'$ as the logarithmic term in the expansion \eqref{eq:Qprime_expansion-intro}.


If $\Omega\subset\mathbb{C}^{n+1}$, we may replace $g$ in the formula \eqref{eq:Qprime_expansion-intro}
by the complete K\"ahler--Einstein metric $g_+$ given by Cheng--Yau and $r$ by the exact solution $u$
to the Monge--Amp\`ere equation.
More generally, if $g_+$ is a globally defined complete K\"ahler--Einstein metric on $\Omega\subset X$
that is given by $dd^c\log u$ with a function $u$ of the form \eqref{uexpansion} near $\pa\Omega$,
then the same replacement is possible, as we shall confirm below.

\begin{lem}\label{deffunlemma}
For $r$ and $u$ as above, let $\omega_+=\omega+dd^c\varphi$,
where $\varphi$ is a function on $\Omega$ agreeing with $\log u-\log r$ near $\pa\Omega$. Then
$$
\int_{u>\epsilon}\pa\log u\wedge\conj\pa\log u\wedge\omega_+^n
=\int_{r>\epsilon}\pa\log r\wedge\conj\pa\log r\wedge\omega^n+O(1)
$$
and
$$
\int_{u>\epsilon}\omega_+^{n+1}=
\int_{r>\epsilon}\omega^{n+1}+O(\epsilon\log\epsilon).
$$
In each equation, one can also replace $u>\epsilon$ by $r>\epsilon$, and vice versa.
\end{lem}
\begin{proof}


We first examine the choice of defining function.  Take a trivialization of
$\Omega$ near the boundary $M\times(0,1)$ such that the second component agrees with $r$. 
In each equality, the integrand is of the form $f(x,r)d\sigma\wedge dr$
with a function $f(x,r)=O(r^{-n-2})$, where $x$ is a coordinate on $M$ and $d\sigma$ is a fixed volume form
on $M$. Thus
$$
\int_{u>\epsilon}f(x,r)d\sigma\wedge dr-
\int_{r>\epsilon}f(x,r)d\sigma\wedge dr
=\int_{M}\int_{r_\epsilon(x)}^\epsilon f(x,r)drd\sigma,
$$
where $r_\epsilon(x)$ is defined by $u(x,r_\epsilon(x))=\epsilon$.
Since $r_\epsilon(x)=\epsilon(1+O(\epsilon^{n+2}\log\epsilon))$,
we have
\begin{align*}
\int_{r_\epsilon(x)}^\epsilon|f(x,r)|dr&=
O\left(\big[r^{-n-1}\big]_{\epsilon(1+O(\epsilon^{n+2}\log\epsilon))}^\epsilon\right)\\
&=\epsilon^{-n-1}O(\epsilon^{n+2}\log\epsilon)
=O(\epsilon\log\epsilon).
\end{align*}
This estimate is uniform over $M$.

We next prove the first equation with $u>\epsilon$ replaced by $r>\epsilon$ on the left-hand side.
Differentiating \eqref{uexpansion}, we have
\begin{align*}
\pa u&=\pa r+O(r^{n+2}\log r),
\\
\pa\conj\pa u&=\pa\conj\pa r+O(r^{n+2}\log r)\mod \pa r
\end{align*}
and thus
$$
\pa u\wedge\conj\pa u\wedge\pa\conj\pa u=
\pa r\wedge\conj\pa r\wedge\pa\conj\pa r+O(r^{n+2}\log r).
$$
It follows that
$$
\pa\log u\wedge\conj\pa\log u\wedge\omega_+^n=
\pa\log r\wedge\conj\pa\log r\wedge\omega^n+O(\log r).
$$
The integration of the difference $O(\log r)$ over $r>\epsilon$ gives $O(1)$. 

To prove the second formula, we write
$$
\omega_+^{n+1}=\omega^{n+1}+dd^c \varphi\wedge\Phi,
$$
where
$\Phi=\sum_{j=0}^{n}\begin{binom}{n+1}{j+1}\end{binom}(dd^c\varphi)^j \wedge\omega^{n-j}$.
Since $\Phi$ is closed, we have
\begin{equation}\label{difference}
\int_{r>\epsilon}\omega_+^{n+1}-\omega^{n+1}=\int_{r=\epsilon}
d^c\varphi\wedge\Phi.
\end{equation}
On the other hand,  $d^c\varphi=O(r^{n+1}\log r)$ and $
\Phi= O(r^{-n})$ mod $dr$
give
$
d^c\varphi\wedge\Phi=O(r\log r)$ mod $dr.
$
Hence the right-hand side of \eqref{difference} is $O(\epsilon\log\epsilon)$.
\end{proof}

\subsection{Proof of Theorem \ref{renormalized-volume-thm}}\label{proofthm1}

For $\theta\in\calPE$, let $r=\bfr/h_\theta^{1/(n+2)}$. Then we may write $\omega_+=\omega+dd^c\varphi$ as
in Lemma \ref{deffunlemma}.
To compute the renormalized volume $V$, $\omega^{n+1}$ can be used in place of $\omega_+^{n+1}$.
Since $\omega$ is given by $dd^c\log r$ near $\pa\Omega$,
\begin{equation}\label{kahler-decomp}
\omega=dd^c\log r+\omega_0,
\end{equation}
where $\omega_0$ has compact support.
The pair $(\omega_0,0)\in\wedge^2(\Omega)\oplus\wedge^1(\pa\Omega)$ defines a relative cohomology class $\wt c_1(K_\Omega)\in H^2(\Omega,\pa\Omega;\bR)$ which projects to $c_1(K_\Omega)\in H^2(\Omega;\bR)$.
Taking the power of \eqref{kahler-decomp}, we get
\begin{align*}
\omega^{n+1}
&=\left(dd^c\log r\right)^{n+1}+\omega_0^{n+1}+d\eta\\
&=
d\Bigl(r^{-1}\vartheta
\wedge \big(d(r^{-1}\vartheta)\big)^{n}\Bigr)+\omega_0^{n+1}+d\eta,
\end{align*}
where  $\vartheta =d^c r$ and $\eta$ is
 a compactly supported $(2n+1)$-form on $\Omega$.
Thus Stokes' theorem implies
\begin{equation}\label{stokes-eq}
\int_{\Omega^\epsilon}\omega^{n+1}=
\int_{\pa\Omega^\e}\frac{\vartheta}{r}
\wedge\Bigl(d\Bigl(\frac{\vartheta}{r}\Bigr)\Bigr)^{n} +
\int_\Omega \omega_0^{n+1}.
\end{equation}
Note that any different choice of $\theta\in\calPE$ gives the same integral of $\omega_0^{n+1}$.
In fact, using Proposition \ref{psudo-Einstein-prop} (2), one can show that
the change of $\omega_0^{n+1}$ is of the form $d\eta'$ with $\eta'$ compactly supported.

For small $\epsilon_0>0$, the set $\conj\Omega\setminus\Omega^{\epsilon_0}$ 
 can be identified with $M\times [0, \e_{0}]$ by the flow generated by $N$
 defined in \eqref{def-xi} with respect to $\rho=r$. We regard $\theta$ as a 1-form on 
this product by pulling it back via the natural projection onto $M$. 
Then we have
\begin{equation}\label{def-a}
\vartheta\wedge (d\vartheta )^n=s(x,r)\theta\wedge (d\theta)^n\mod dr
\end{equation}
for a function $s(x,r)$.
Hence
$$
\int_{\pa\Omega^\e}\frac{\vartheta}{r}
\wedge\Bigl(d\Bigl(\frac{\vartheta}{r}\Bigr)\Bigr)^{n}
=\epsilon^{-n-1}
\int_{M}s(x,\epsilon)\theta\wedge (d\theta)^n.
$$
Comparing  the constant terms in \eqref{stokes-eq}  gives
\begin{equation}
(n+1)!V=\int_M \frac{s^{(n+1)}}{(n+1)!}\theta
\wedge(d\theta)^n+\int_\Omega \omega_0^{n+1},
\end{equation}
where we have set  $s^{(j)}=N^j s|_M$.

On the other hand, near the boundary, we have
\begin{equation*}
	\begin{split}
		|d\log r|^2_gdv_g
		&=-2i\pa\log r\wedge\ol\pa\log r\wedge\frac{\omega^{n}}{n!}
		=\frac{2}{n!}\frac{dr\wedge\vartheta\wedge(d\vartheta)^n}{r^{n+2}}\\
		&=\frac{2}{n!}\frac{s(x,r)}{r^{n+2}}\,dr\wedge\theta\wedge(d\theta)^n.
	\end{split}
\end{equation*}
Thus comparing the coefficient of $\log\epsilon$ in \eqref{eq:Qprime_expansion-intro} 
gives
\begin{equation}\label{totalQ-sn}
\overline Q'=\frac{(-1)^{n+1}2\cdot n!}{n+1}\int_M s^{(n+1)}\theta\wedge(d\theta)^n.
\end{equation}
This completes the proof of the theorem.
\qed

\section{Hamiltonian flow on the ambient space}\label{Hamilton-sec}
\subsection{Hamiltonian flow}
Let $M_t$, $t\in (-1,1)$, be a smooth family of strictly pseudoconvex real hypersurfaces in a complex manifold $X$.  Setting $\calN_t=K_{M_t}^*$, we take
a smooth family of Fefferman defining functions $\bfr_t$ of $\calN_t$ and denote the family of ambient metrics by $\wt g_t=\wt g[\bfr_t]$. For $t=0$, we will omit the subscript $0$. We use `dot' for the derivative in $t$, e.g.~ $\dot\bfr_t=(d/dt)\bfr_t$.

For a smooth family of complex valued functions 
$F_t\in\wt\calE(1)$, 
we define a time-dependent real vector field on $\wt X$, the gradient of $F_t$, by
$$
Y_t=\re\big( (F_t)^A\pa_{A}\big),
$$
where the index is raised by using $\wt g_t$.
Let $\Phi_t$ be the isotopy generated by $Y_t$, i.e., $\Phi_t$
is a smooth family of diffeomorphisms in a neighborhood of $\calN$ such that
$$
\frac{d}{dt}\Phi_t(p)=Y_t(\Phi_t(p)).
$$
Since ${F}_t\in\wt\calE(1)$, the vector field $Y_t$ is invariant under dilations $\delta_\lambda$. Hence $\Phi_t$ is a bundle map of $\wt X$, i.e., $\delta_\lambda\circ\Phi_t=\Phi_t\circ\delta_\lambda$. We denote the induced flow on $X$ by $\underline{\Phi}_t$ and the generator of $\underline\Phi_t$ by $\underline Y_t$. Note that  $\underline Y_t$ is characterized by $\underline Y_t=\pi_* Y_t$, where $\pi\colon\wt X\to X$ is the projection.

\begin{lem}
	\label{hamiltonian-flow-lemma}
	Let $\boldsymbol{\vartheta}_t=d^c\bfr_t$. 
	If $\re F_t=-\dot\bfr_t,$
then
$$
\bfr_t\circ\Phi_t=\bfr,\qquad
\Phi^*_t(\boldsymbol{\vartheta}_t)=\boldsymbol{\vartheta}.
$$
In particular, $\underline\Phi_t\colon M\to M_t$ is a contact diffeomorphism.
\end{lem}

\begin{proof}
It suffices to verify that  the derivatives in $t$ vanish.
For the first equation,
$$
\frac{d}{dt}(\bfr_t\circ\Phi_{t})
=\Phi_{t}^*(Y_t\bfr_t+\dot\bfr_t)
=\Phi_{t}^*(\re F_t+\dot\bfr_t)=0.
$$
Here we have used the fact $F_t\in\wt\calE(1)$ to derive $F_t^A\pa_A\bfr_t=F_t$.
To prove the second one, we first note
\begin{align*}
	Y_t\CR d\boldsymbol{\vartheta}_t
	&=-\frac{i} {2}\conj\pa F_t+\frac{i} {2}\pa\conj F_t
	=d^c\re F_t+\frac12d\im F_t,
	\\
	\boldsymbol{\vartheta}_t(Y_t)&=
	\frac{i} 4\big((\bfr_t)_A(F_t)^A -(\bfr_t)_{\conj A}(\conj F_t)^{\conj A}\,\big)
	=-\frac{1}2\im F_t
\end{align*}
and $\dot{\boldsymbol{\vartheta}}_t=d^c\dot\bfr_t=-d^c\re F_t$.
Then Cartan's formula gives
\begin{equation*}
(\Phi_t^{-1})^*\frac{d}{dt}\Phi^*_t(\boldsymbol{\vartheta}_t)
=\calL_{Y_t}\boldsymbol{\vartheta}_t+\dot{\boldsymbol{\vartheta}}_t
= Y_t\CR d\boldsymbol{\vartheta}_t+d\big(\boldsymbol{\vartheta}_t( Y_t)\big)+\dot{\boldsymbol{\vartheta}}_t=0.
\qedhere
\end{equation*}
\end{proof}

One can show that $Y=Y_0$ agrees with the vector field
introduced in \cite{AGL} in the description of the deformation complex
by computing it in terms of the Levi metric for a fixed $\theta$.  
Using the matrix expression \eqref{wtgatN} of $\wt g$ in the frame $Z_A$, we have on $\calN$
\begin{equation*}
Y=
\re\big(\wt g^{A\conj B}(Z_{\conj B}F)Z_A\big)
=|z^0|^{-2}\re\big((\kappa F+\conj\xi F) Z_{0}+F\xi-
F^{\alpha}Z_{\alpha}
\big).
\end{equation*}
Hence, writing $F=|z^0|^2\underline F$ with a function $\underline F$ on $X$, we obtain
\begin{equation}\label{uY-exp}
\underline Y=\pi_*|z^0|^{-2}\re(F\xi-
 F^{\alpha}Z_{\alpha})
=
\re(\underline F\xi-
\underline F^{\alpha}Z_{\alpha}).
\end{equation}
It shows that the map $F\mapsto \underline Y$ agrees with the first operator $D$ of the deformation complex given in \S4 of \cite{AGL}.
In particular, if $F$ is pure imaginary, then setting $\underline F=i u$ and using
$\xi=N+(i/2)T$, we get
$$
\underline Y=-\frac12 u\,T-\re i\, u^\alpha Z_\alpha, 
$$ 
which agrees with the infinitesimal contact diffeomorphism
studied in \cite{CL}.

\subsection{Relation to the deformation complex}\label{deformation-complex}
Now we assume $\re F_t=-\dot\bfr_t$ as in the lemma above. Then
$\underline\Phi_t\colon M\to M_t$ give contact diffeomorphisms and we may compare the CR structures via the pull-back by $\underline\Phi_t$.

As in the definition of the Graham--Lee connection, 
let $Z_{\alpha}$ be sections of $T^{1,0}X$ near $M$ that give
a frame of $T^{1,0}M$ and set $Z_{n+1}=\xi$ so that $Z_\alpha, Z_{n+1}$ span $T^{1,0}X$.
For small $\abs{t}$, we may give a frame of $T^{0,1}M_t$ by
$$
Z_\calpha^t=Z_\calpha+\varphi_\calpha(t)\conj\xi,
$$
where $\varphi_\calpha(t)=-Z_\calpha\bfr_t/{\conj\xi}\bfr_t$.
Let $\underline\Phi_{t}^*Z^t_\calpha=(\underline\Phi_{t}^{-1}){}_*Z^t_\calpha$.
Then, since $\underline\Phi_t$ preserves the contact structure, there exists a tensor $\varphi_{\calpha\cbeta}\in\calE_{\calpha\cbeta}(1)$ such that
$$
\underline\Phi_{t}^*Z^t_\calpha=(1+O(t))Z_\calpha+t\varphi_\calpha{}^{\beta}Z_{\beta}+O(t^2)
$$
as $t\to 0$.
Here the index on $\varphi$ is raised by using the Levi metric $\boldsymbol{l}^{\alpha\conj\beta}$ of $M$.  We shall compute $\varphi_{\calpha\cbeta}$.

\begin{lem}
Let $F=F_0$. Then
$\varphi_{\calpha\cbeta}$ depends only on $f=F|_\calN\in\calE(1)$ and $f\mapsto\varphi_{\calpha\cbeta}$ defines a
CR invariant differential operator
$
P_{\calpha\cbeta}\colon\calE(1)\to\calE_{\calpha\cbeta}(1),
$
which is given in terms of the Tanaka--Webster connection by
\begin{equation}\label{Pab-expression}
2P_{\calpha\cbeta}f=
\nabla_{\conj\alpha\conj\beta}f-iA_{\conj\alpha\conj\beta}f.
\end{equation}
\end{lem}
\begin{proof}
By the definition, $\varphi_{\conj\alpha}{}^{\beta}$ is given by the derivative
$$
\delta_0(\underline\Phi_{t}^*Z^t_{\conj\alpha})=\varphi_{\conj\alpha}{}^{\beta}Z_{\beta}
\mod T^{0,1} X,
$$
where $\delta_0=(d/dt)|_{t=0}$.
On the other hand, noting that $\delta_0 Z_\calpha^t=\dot\varphi_\calpha\,{\conj\xi}\in
 T^{0,1}\wt X$, we have
$$
\delta_0(\underline\Phi_{t}^*Z^t_\calpha)
=[\uY,Z_\calpha]\mod T^{0,1} X.
$$
Since $\underline Y$ along $M$  is determined by $f\in\calE(1)$,
we see that $f\mapsto\varphi_{\calpha\cbeta}$ is a CR invariant differential operator.

To compute $[\uY,Z_\calpha]$, we use \eqref{uY-exp}, which gives
$$
2\uY=\uF\xi-\uF^{\alpha}Z_{\alpha}\mod T^{0,1} X.
$$
It follows from the structure equation \eqref{GLstr} that
\begin{align*}
2[\uY,Z_{\conj\alpha}]&=[\uF\xi-
\uF^{\beta}Z_{\beta},Z_{\conj \alpha}]\mod T^{0,1} X\\
&
=\uF[\xi,Z_{\conj \alpha}]-[
\uF^{\beta}Z_{\beta},Z_{\conj \alpha}]\mod T^{0,1} X \text{ and }\xi\\
&
=-i\uF A_{\conj \alpha}{^\beta}Z_\beta+
(\nabla_{\conj\alpha}\,\uF^{\beta})Z_{\beta}\mod T^{0,1} X \text{ and }\xi. 
\end{align*}
Thus we get  $2\varphi_{\conj\alpha\conj\beta}=\nabla_{\conj\alpha\conj\beta}F-iA_{\conj\alpha\conj\beta}F$,
which gives \eqref{Pab-expression}.
In this computation, we have used an admissible frame with respect to a contact form normalized by $dz^1\wedge\cdots\wedge dz^{n+1}$. However, one can easily check that $\nabla_{\conj\alpha\conj\beta}f-iA_{\conj\alpha\conj\beta}f$, $f\in\calE(1)$, is invariant under the scaling and 
\eqref{Pab-expression} holds for any $\theta$.
\end{proof}

We can also write $P_{\conj\alpha\conj\beta}$ in terms of the ambient metric $\wt g$.
For  $\wt g$, choose a frame $Z_A$ of $T^{1,0}\wt X$ as in \S\ref{sec-ambient-metric}. Since $\wt\nabla$ is torsion-free and preserves the complex structure, we have
\begin{align*}
[Y, Z_\calpha]&
=\wt\nabla_{Y}Z_{\calpha}-\wt\nabla_{Z_\calpha}Y \\
&=-\frac12(\wt\nabla{}_\calpha{}^{B} {F} )Z_{B} 
\mod T^{0,1}\wt X.
\end{align*}
Since $\pi_*([Y, Z_\calpha])=[\underline Y, Z_\calpha]$, 
 we have $2\varphi_\calpha{}^\beta=-\wt\nabla{}_\calpha{}^{\beta}{F}$.
Lowering the index $\beta$ of both sides  by using $\boldsymbol{l}_{\alpha\conj\beta}$ and $\wt g_{\alpha\cbeta}=-\boldsymbol{l}_{\alpha\cbeta}$ respectively, we have
$
\varphi_{\calpha\cbeta}=2\wn_{\calpha\cbeta}F|_{\calN}.
$
Note that $\wn_{\calpha\cbeta}F|_{\calN}$ depends only on $f$. If ${F}=\bfr \psi$ with $\psi\in\wt\calE(0)$, then
$$
\wt\nabla_{\cA\,\cB}{F}=\bfr_\cA\psi_\cB+\bfr_\cB\psi_\cA+O(\bfr)
$$
because $\bfr_{\cA\,\cB}=0$, which is 
the conjugate of \eqref{rhoAB}.
Since $\conj\pa\bfr$ vanishes on $T^{0,1}\calN$, the restriction
of $\wt\nabla_{\cA\,\cB}{F}$ vanishes on $\otimes^2 T^{0,1}\calN$. 
Moreover, since
$$
\bfr^\cA\wn_{\cA\,\cB} F=\wn_{\cB}(\bfr^\cA F_\cA)-\tensor{\bfr}{^\cA_\cB}F_\cA=0,
$$
 we see that $\wt\nabla_{\calpha\cbeta}F|_\calN\in\calE_{\calpha\cbeta}(1)$ is well-defined.  This gives another proof of the CR invariance of 
$
P_{\calpha\cbeta}.
$

Note that $P_{\calpha\cbeta}$ is the first operator in the deformation complex of CR structures, which can also be seen as a subcomplex in the BGG sequence constructed by \cite{CSS}; see also \cite{AGL}.
From this point of view, we can compare $P_{\calpha\cbeta}$ with
$\conj\pa_b$ operator on functions, which is the first operator in the BGG sequence of the trivial representation.

The operator $\conj\pa_b$ can be seen as the restriction to $\calN$ of $\conj\pa$ on $\wE(0)$.  A function $f\in\calE(0)$ on $\calN$
can be extended to a holomorphic function on the pseudoconvex side of $\calN$ if and only if $\conj\pa_b f=0$.
We have analogous result for $P_{\calpha\cbeta}$ and $\wn_{\cA\,\cB}$ on $\wE(1)$.  It is clear that if $\wn_{\cA\,\cB}F=0$, then $f=F|_\calN$ satisfies $P_{\calpha\cbeta}f=0$. We can prove the converse in a formal sense.  To simplify the notation, we set $P_{\alpha\beta}f=\conj{P_{\calpha\cbeta}\conj f}$ and consider the equation $P_{\alpha\beta}f=0$.

\begin{prop}\label{ext-prop} 
If $f\in\calE(1)$ satisfies $P_{\alpha\beta}f=0$, then there exists an  extension $F\in\wE(1)$ of $f$ such that $\wt\nabla_{AB}F=O(\bfr^\infty)$.
\end{prop}

\begin{proof}
Let $F$ be an extension of $f$ such that $\wDelta F=O(\bfr^2)$; see
\eqref{harmonic-ext} for a construction of $F$.
Since $\bfr^A\wn_{AB} F=0$
and $\wn_{\alpha\beta} F|_\calN=2P_{\alpha\beta} f=0$, we may write
\begin{equation}
	\label{BGG-extension-1}
	\wn_{AB} F=2\bfr_{(A}\varphi_{B)}+\bfr\varphi_{AB}
\end{equation}
with contravariant tensors $\varphi_A$ and $\varphi_{AB}$ on $\wt X$  that are invariant under $\delta_\lambda^*$.
From $\bfr^A\wn_{AB} F=0$, we have
$
0=\bfr_B \bfr^A\varphi_A+\bfr ( \varphi_B+\bfr^A\varphi_{AB})
$
and hence
$$
\varphi_B+\bfr^A\varphi_{AB}=0\mod \bfr_B.
$$
Applying $\wn^A$ to \eqref{BGG-extension-1},
we get
\begin{align*}
\wn^A{}_{AB} F&=\wn^{A}(\bfr_A\varphi_B)+\wn^A(\bfr_B\varphi_A)+\bfr^A\varphi_{AB}+O(\bfr)\\
&=(n+3)\varphi_B+\bfr^A\varphi_{AB}+O(\bfr)\mod \bfr_B
\\
&=(n+2)\varphi_B+O(\bfr)\mod \bfr_B. 
\end{align*}
On the other hand, the left-hand side gives
$$
\wn^A{}_{AB} F=\wn_B\wt\Delta F+\Ric_B{}^A\wn_A F=O(\bfr).
$$
It follows that
$\varphi_B=O(\bfr)$ modulo $\bfr_B$.
Thus we may write
$$
\wn_{AB} F=\bfr_{A}\bfr_B\varphi+\bfr\varphi_{AB}
$$
with $\varphi\in\wE(-1)$.
Then $\bfr^A\wn_{AB} F=0$ and $\wn^A{}_{AB} F=O(\bfr)$ respectively give
\begin{align*}
\bfr_B\varphi+\bfr^A\varphi_{AB}&=0,
\\
(n+2)\bfr_B\varphi+\bfr^A\varphi_{AB}&=O(\bfr).
\end{align*}
Therefore, we get $\bfr_B\varphi=O(\bfr)$ and so $\wn_{AB} F=O(\bfr).
$ 
Now we set
$$
\wn_{AB} F=\bfr^m\varphi_{AB}
$$
for an integer $m\ge1$; we know the existence of such an $F$ when $m=1$. Then
$$
0=\wn_{[CA]B} F=m\bfr^{m-1}\bfr{}_{[C}\varphi{}_{A]B}+O(\bfr^{m})
$$
implies $\bfr{}_{[C}\varphi{}_{A]B}=O(\bfr)$.  Since $\varphi_{AB}$ is symmetric, we may write
$$
\varphi_{AB}=\bfr_A\bfr_B \psi+O(\bfr)$$
with $\psi\in\wt\calE(-m-1)$.
Hence, replacing $ F$ by $ F-\bfr^{m+2}\psi/(m+2)(m+1)$, we 
have
$
\wn_{AB} F=O(\bfr^{m+1}).
$
By induction on $m$, we may find an $F$ such 
that $\wn_{AB} F=O(\bfr^m)$ for each $m$. Thus Borel's lemma gives an $F$ satisfying $\wn_{AB} F=O(\bfr^\infty)$.
\end{proof}

\begin{rem}\rm
The extension $F$ in the proposition above is not unique even formally.
In fact, we have $\wt\nabla_{AB}(\bfr \psi)=0$ for any antiholomorphic $\psi\in\wt\calE(0)$.
One can easily see that this is the only ambiguity allowed in the extension $F$ of $f$ satisfying $\wt\nabla_{AB}F=0$.
\end{rem}

\section{The variational formula of total $Q$-prime curvature}

\subsection{Reduction to the variation of the potential function}
Let $M_t$ be a family of compact strictly pseudoconvex real hypersurfaces in $X$
defined by $\rho_t\in C^\infty(X)$ such that $\rho(z,t)=\rho_t(z)$ is smooth in $(z,t)\in X\times (-1,1)$.
We make no further assumption on the function $\rho_t$ in this subsection.
We will write $\Omega_t=\set{\rho_t>0}$ and $\Omega^\epsilon_t=\set{\rho_t>\epsilon}$.
For $t=0$, the subscript $0$ will be omitted.

Let $g_t$ be the K\"ahler metric, defined on a one-sided neighborhood of $M_t$, associated to
$dd^c G_t$, where $G_t=\log\rho_t$.
The volume form and the squared norm for $g_t$ are denoted by $\vol_t$ and $|\cdot|^2_t$.
We write
\begin{equation*}
	\int_{\Omega^\epsilon_t}\abs{dG_t}_{t}^2\vol_{t}
	=w^{(0)}_t\epsilon^{-n-1}+\dots+w^{(n)}_t\epsilon^{-1}+L_t\log\epsilon+R_t^\epsilon,
\end{equation*}
where $R_t^\epsilon$ as well as its $t$-derivative is uniformly $O(1)$ as $\epsilon\to 0$.
By differentiating this at $t=0$, we obtain
\begin{equation*}
\delta_0\int_{\Omega^\epsilon_t}\abs{dG_t}_{t}^2\vol_{t}
	=\dot w^{(0)}\epsilon^{-n-1}+\dots+\dot w^{(n)}\epsilon^{-1}
	+\dot L\log\epsilon+O(1).
\end{equation*}
Since we are interested in the coefficient of $\log\epsilon$, it is useful to introduce the notation
$\lp$:
$$
\lp F(\epsilon)=c_0\quad\text{when}\quad
F(\epsilon)=\sum_{j=1}^{N}c_j\epsilon^{-j}+c_0\log\epsilon+O(1).
$$ 
Then we may write
\begin{equation}
	\label{eq:lp_part_of_time_derivative_of_dG_integral}
	\lp\delta_0\int_{\Omega^\epsilon_t}\abs{dG_t}_{t}^2\vol_{t}=\dot L.
\end{equation}

\begin{prop}\label{dot_L_prop}
Let $\dot G=\delta_0 G_t$. Then
\begin{equation}
	\label{eq:variation_of_L_and_antegral_of_u}
	\dot L=2(n+1)(n+2)\lp \int_{\Omega^\epsilon}\dot G\vol.
\end{equation}
\end{prop}

To prove the proposition, we break the variation on the left-hand side of
\eqref{eq:lp_part_of_time_derivative_of_dG_integral} into two parts.
The first one is the contribution of the change of the integrand,
and the second one is that of the wiggle of the boundary $M^\epsilon=\pa\Omega^\epsilon$:
\begin{equation}
	\label{eq:breaking_total_variation}
	\begin{split}
		\delta_0\int_{\Omega^\epsilon_t}\abs{dG_t}_t^2\vol_t
		&=\int_{\Omega^\epsilon}
		\delta_0(\abs{dG_t}_t^2\vol_t)
		+\lim_{t\to 0}\frac{1}{t}\left(\int_{\Omega^\epsilon_t}\abs{dG_t}_t^2\vol_t
		-\int_{\Omega^\epsilon}\abs{dG_t}_t^2\vol_t\right)\\
		&=\int_{\Omega^\epsilon}
		\delta_0(\abs{dG_t}_t^2\vol_t)
		+\lim_{t\to 0}\frac{1}{t}\left(\int_{\Omega^\epsilon_t}\abs{dG}^2\vol
		-\int_{\Omega^\epsilon}\abs{dG}^2\vol\right)\\
		&=:\intf+\ints.
	\end{split}
\end{equation}
Note that, when $\epsilon>0$ is fixed, the integral of $\abs{dG_t}^2_t\vol_t$ over $\Omega^\epsilon$ and
that of $\abs{dG}^2\vol$ over $\Omega^\epsilon_t$ make sense for sufficiently small $\abs{t}$.
The second equality of \eqref{eq:breaking_total_variation} is because
$\abs{dG_t}^2_t\vol_t$ is uniformly convergent to $\abs{dG}^2\vol$ in
a neighborhood of $\cl{\Omega^\epsilon}$.

The first variation of $\abs{dG}^2\vol=2\tensor{G}{^a}\tensor{G}{_a}\vol$ is
\begin{equation*}
	2(-\tensor{\dot{g}}{^a^{\conj{b} }}\tensor{G}{_a}\tensor{G}{_{\conj{b} }}
	+2\re(\tensor{G}{^a}\tensor{\dot{G}}{_a})
	+\tensor{G}{^a}\tensor{G}{_a}\tensor{\dot{g}}{_b^b})\vol.
\end{equation*}
If we write $u=\dot{G}$, then
\begin{equation*}
	\intf=2\int_{\Omega^\epsilon}(\tensor{u}{^a^{\conj{b} }}\tensor{G}{_a}\tensor{G}{_{\conj{b} }}
	+2\re(\tensor{G}{^a}\tensor{u}{_a})-\tensor{G}{^a}\tensor{G}{_a}\tensor{u}{_b^b})\vol.
\end{equation*}
We integrate this by parts. We first observe that there is no contribution from the boundary term to the log term.

\begin{lem}\label{divergence-lem}
Let $f_adz^a$ be a $(1,0)$-form on $\Omega$ such that $\rho^m f_a dz^a$ extends smoothly up to the boundary for some integer $m$. Then
$$
\lp\int_{\Omega^\epsilon}\nabla^af_a\,\vol=0.
$$
\end{lem}
\begin{proof}
	Since $d(f_adz^a\wedge\omega^n)=-(n+1)^{-1}\nabla^a f_a\,\omega^{n+1}$, we have
$$
\int_{\Omega^\epsilon}\nabla^a f_a\,\omega^{n+1}=-(n+1)\int_{\pa\Omega^\epsilon}f_a dz^a\wedge\omega^n.
$$
The right-hand side does not contain any $\log\epsilon$ term.
\end{proof}

Applying this lemma repeatedly, we have
\begin{align*}
\lp\int_{\Omega^\epsilon}&(\tensor{u}{^a^{\conj{b} }}\tensor{G}{_a}\tensor{G}{_{\conj{b} }}
	+2\re(\tensor{G}{^a}\tensor{u}{_a})-\tensor{G}{^a}\tensor{G}{_a}\tensor{u}{_b^b})\vol
\\
&=
\lp\int_{\Omega^\epsilon}
u\,(\nabla^a {}_b (\tensor{G}{_a}\tensor{G}{^b})
	-2\re(\tensor{G}{^a_a})
	-\nabla^b{}_b(\tensor{G}{^a}\tensor{G}{_a}))\vol.
\end{align*}
Since $G_{a\conj{b} c}=0$ and $\tensor{G}{_a^a}=-(n+1)$, the integrand of the last formula is reduced to
\begin{equation*}
\begin{split}
u\,&(\tensor{G}{_a^a}\tensor{G}{^b_b}+\tensor{G}{_a_b}\tensor{G}{^a^b}+2(n+1)
-\tensor{G}{_a_b}\tensor{G}{^a^b}-\tensor{G}{_a^b}\tensor{G}{^a_b})
\\
&=
u\,((n+1)^2+2(n+1)-(n+1))
\\
&=(n+1)(n+2)u.
\end{split}
\end{equation*}
Thus we conclude 
$$
\lp \intf=2(n+1)(n+2)\lp \int_{\Omega^\epsilon}u\vol.
$$

To compute $\ints$, as in \S\ref{proofthm1},
we use the identification between a neighborhood of $M$ with the product $M\times(-\epsilon_0,\epsilon_0)$
given by the flow of the normal vector field $N$.
Then $\pa\Omega^\epsilon$ is identified with $M\times\{\epsilon\}$, while $\pa\Omega_t^\epsilon$ is given by $\{(x,\psi(x,t,\epsilon)):x\in M\}$ for a smooth function $\psi$.
Fix a volume form $d\sigma$ on $M$ and write $\abs{dG}^2\vol=f(x,\rho)\,d\rho d\sigma$.
Then $\ints$ can be expressed by the derivative of an iterated integral:
\begin{equation*}
\begin{split}
	\ints&=\delta_0	\int_M\int_{\psi(x,t,\epsilon)}^{\epsilon}
	f(x,\rho)\,d\rho\,d\sigma+O(1)\\
&=
	-\int_M\delta_0\psi(x,t,\epsilon)
	f(x,\epsilon)\,d\sigma+O(1).
\end{split}
\end{equation*}
Since there is no $\log\epsilon$ term contained in the integrand, we have
$\lp \ints=0$.
This completes the proof of the proposition.

\subsection{Proof of Theorem \ref{variation_thm}: the first variation}
We apply Proposition \ref{dot_L_prop} to the case $\rho_t= r_t$,
where $r_t=\bfr_t/h^{1/(n+2)}$ is the Fefferman defining function of $\Omega_t$ associated to
a smooth family $\bfr_t$ and the flat metric $h$ of the canonical bundle $K_X$ in the assumption of
Theorem \ref{variation_thm}.
In view of \eqref{eq:variation_of_L_and_antegral_of_u}, we need the relation between 
$\dot G=\dot\bfr/\bfr$ and $\calO$.
We first compute the equation satisfied by $\dot\bfr$.

\begin{lem} \label{cor-dotr1}
	It holds that
	\begin{equation}\label{varphieq}
		\wt\Delta\dot\bfr=(n+2)\dot\bfr \calO\bfr^{n+1}+\dot\calO\bfr^{n+2}+O(\bfr^{2n+3}),
	\end{equation}
	where $\wt\Delta$ is the Laplacian of $\wt g_0$.
\end{lem}
\begin{proof}
We take
the first variation in $t$ of the Monge--Amp\`ere equation
\begin{equation}\label{MA-eq-t}
(dd^c\bfr_t)^{n+2}=(1+\calO_t\bfr_t^{n+2})\wt\vol.
\end{equation}
The left-hand side has expansion
\begin{align*}
(dd^c\bfr_t)^{n+2}&
=(dd^c\bfr)^{n+2}\big(1+t\wt\Delta\dot\bfr\big)+O(t^2)\\
&=(1+\calO\bfr^{n+2})\big(1+t\wt\Delta\dot\bfr\big)\wt\vol+O(t^2)
\end{align*}
while the expansion of the right-hand side gives
\begin{align*}
1+\calO_t\bfr_t^{n+2}&=1+(\calO+t\dot\calO)(\bfr+t\,\dot\bfr)^{n+2}+O(t^2)\\
&=1+\calO\bfr^{n+2}+
t\big(\dot\calO\bfr^{n+2}+(n+2)\dot\bfr \calO\bfr^{n+1}\big)+O(t^2).
\end{align*}
Thus comparing the coefficients of $t$, we have
\begin{equation*}
(1+\calO\bfr^{n+2})\wt\Delta\dot\bfr=(n+2)\dot\bfr \calO\bfr^{n+1}+\dot\calO\bfr^{n+2}
\end{equation*}
and the lemma follows.
\end{proof}

In particular, the lemma above gives
\begin{equation*}
\wt\Delta\dot\bfr=(n+2)\dot\bfr \calO\bfr^{n+1}+O(\bfr^{n+2}).
\end{equation*}
On the other hand, since $\bfr$, $\dot\bfr\in\wt\calE(1)$ implies
$\bfr^A\bfr_A=\bfr$ and $\bfr^A\dot\bfr_A=\dot\bfr$, we have
\begin{align*}
	\bfr\wt\Delta \dot{G}
&=\bfr\wt\Delta (\dot\bfr/\bfr)\\
&=\wt\Delta\dot\bfr-(n+2)\dot\bfr/\bfr-
(\bfr^A\dot\bfr_A
+\bfr_A\dot\bfr^A)/\bfr+2\bfr^A \bfr_A\dot\bfr/\bfr^2\\
&=\wt\Delta\dot\bfr-(n+2)\dot{G}.
\end{align*}
Thus we get the identity
$
(\bfr\wt\Delta+n+2)\dot G=\wt\Delta\dot\bfr
$
and hence
$$
(\bfr\wt\Delta+n+2)\dot G
=(n+2)\dot\bfr \calO\bfr^{n+1}+O(\bfr^{n+2})
$$
follows.

Let $\Delta=\tensor{\nabla}{_b^b}$ be the K\"ahler Laplacian of $g$.
Then \cite[Prop.\ 5.4]{GG} shows that
$
\Delta u=-\bfr\wt\Delta u
$
for any $u\in C^\infty(\Omega)$; here $u$ and $\Delta u$
are identified with a function in $\wE(0)$. 
It follows that
\begin{equation*}
	\label{eq:approx_linearized_equation}
	(-\Delta+n+2)\dot G=(n+2)\dot{\bfr}\mathcal{O}\bfr^{n+1}+O(r^{n+2}).
\end{equation*}
This equation enables us to express $\dot L$ in terms of $\calO$.
Since
\begin{equation*}
	(n+2) \dot G=\Delta \dot G+(n+2)\dot{\bfr}\mathcal{O}\bfr^{n+1}+O(r^{n+2}),
\end{equation*}
we have
\begin{equation}
	\label{eq:integral_of_u}
	(n+2)\int_{\Omega^\epsilon}\dot G\vol
	=\int_{\Omega^\epsilon}\Delta \dot G\vol+(n+2)\int_{\Omega^\epsilon}\dot{\bfr}\mathcal{O}\bfr^{n+1}\vol+O(1).
\end{equation}
By Lemma \ref{divergence-lem}, the first term on the right-hand side
does not contribute to the log term.
If ${\vartheta}=d^c r$, then the volume form of $g$ is given by
\begin{equation*}
	\vol=\frac{\omega^{n+1}}{(n+1)!}
	=-\frac{1}{n!}\frac{1}{r^{n+2}}dr\wedge{\vartheta}\wedge(d{\vartheta})^n
	+\frac{1}{(n+1)!}\frac{1}{r^{n+1}}(d\vartheta)^{n+1}.
\end{equation*}
Hence
\begin{equation*}
	\int_{\Omega^\epsilon}\dot{\bfr}\mathcal{O}\bfr^{n+1}\vol
	=\frac{1}{n!}\log\frac{1}{\epsilon}\int_M\dot{\bfr}\mathcal{O}\theta\wedge(d\theta)^n
	+O(1).
\end{equation*}
Combining this equality with \eqref{eq:variation_of_L_and_antegral_of_u} and \eqref{eq:integral_of_u},
we conclude
\begin{equation*}
	\dot L=-\frac{2(n+1)(n+2)}{n!}\int_M\dot{\bfr}\calO.
\end{equation*}
Thus, by using \eqref{eq:Qprime_expansion-intro}, we obtain \eqref{QPvar}.
	
\subsection{Proof of Theorem \ref{variation_thm}: the second variation}
\label{proof-second-var}
We take the variation of
\begin{equation}
	\label{derivative-of-total-Qprime}
	\delta_t\overline{Q}'_t=c_n\int_{M_t}\dot{\bfr_t}\calO_t,
\end{equation}
where $\dot{\bfr_t}=(d/ds)|_{s=t}\bfr_s$. 
For this, we pull-back the integrand to $M$ by the flow $\underline{\Phi}_t$
given by $F_t=-\dot{\bfr}_t$.

For $\varphi_t\in\wt\calE(-n-1)$ with a parameter $t$, we write
$$
\wt\varphi_t=\varphi_t\boldsymbol{\vartheta}_t\wedge (d\boldsymbol{\vartheta}_t)^n,
$$
where $\boldsymbol{\vartheta}_t=d^c\bfr_t$.
We define a $(2n+1)$-form on $M_t$ by $(\varphi_t)_{M_t}=\iota_t^*\zeta^*\wt\varphi_t$ for a smooth section $\zeta$ of $\wt X\to X$ and an embedding $\iota_t\colon M_t\to X$
(as we have seen in the end of \S\ref{sec-ambient-metric},
this definition is independent of the choice of $\zeta$).
Since Lemma \ref{hamiltonian-flow-lemma} implies $\Phi_t^*\wt{\varphi_t}=\wt{\Phi_t^*\varphi_t}$,
one has
$$
(\Phi_t^*\varphi_t)_M=\underline\Phi_t^*((\varphi_t)_{M_t}).
$$

If we take $\varphi_t=\dot{\bfr_t}\calO_t$,
then the integral on the right-hand side of \eqref{derivative-of-total-Qprime} is
the integral of $\underline{\Phi}_t^*((\varphi_t)_{M_t})$ over $M$.
Therefore, it suffices to compute the first variation of $\Phi_t^*(\dot \bfr_t\calO_t)$, regarded as a function
on the ambient space $\wt{X}$, along $\calN$. It is computed as follows:
$$
\delta_0\big((\dot \bfr_t\calO_t)\circ\Phi_t\big)
=\dot\bfr
\delta_0\big(\calO_t\circ\Phi_t\big)
+\calO
\delta_0\big(\dot\bfr_t\circ\Phi_t\big).
$$
We will show in the appendix that
$$
\delta_0\big(\calO_t\circ\Phi_t\big)=
k_nP_{n+3}
\dot\bfr
$$
for a CR invariant differential operator $P_{n+3}\colon\calE(1)\to\calE(-n-2)$.
On the other hand, we have
$$
\delta_0\big(\dot\bfr_t\circ\Phi_t\big)
=\ddot\bfr+Y\dot\bfr=\ddot\bfr-\dot\bfr_A\, \dot\bfr^A.
$$
Thus we conclude
$$
\delta_0\big((\calO_t\,\dot \bfr_t)\circ\Phi_t\big)
=k_n\dot\bfr
P_{n+3}
\dot\bfr
+\calO\big(\ddot\bfr-\dot\bfr_A\, \dot\bfr^A\big),
$$
which completes the proof of the theorem.

\section{Explicit formulas in dimension $5$}
\label{explicit-formula-section}

\subsection{Computation of the total $Q$-prime curvature}
We give an explicit formula of the total $Q$-prime curvature for $n=2$. Fix a pseudo-Einstein contact form $\theta$ and let $r=\bfr/h_\theta^{1/(n+2)}$ be the associated Fefferman defining function.
We define the sub-Laplacian by $\Delta_b=\ell^{\a\conj{\b}}\nabla_{\a\conj{\b}}$.
By \eqref{totalQ-sn}, the total $Q$-prime curvature is given by a constant multiple of the integral of $s^{(3)}=N^3s|_M$, so we will compute $s^{(3)}$ modulo divergence terms.
Recall from \cite{S} that the function $s$ satisfies the equation
\begin{equation*}
Ns=-3\kappa s,\qquad
s|_M=1,
\end{equation*}
where $\kappa$ is the transverse curvature of $r$.
Thus, setting $\kappa^{(k)}=N^k\kappa|_M$, we have
\begin{equation}\label{s3}
s^{(3)}=-3\kappa^{(2)}+27\kappa^{(0)}\kappa^{(1)}-27(\kappa^{(0)})^3.
\end{equation}
We compute $\kappa^{(j)}$ by employing the methods of \cite{BE2} and \cite{GL}.
Let $\psi_a{}^b$ and $\Psi_a{}^b$ be the connection and the curvature of the metric $g$ given by $dd^c\log r$ near $M$.  Subtracting explicitly the given singular parts from them, we define the smooth parts by
\begin{equation*}
\theta_a{}^b =\psi_a{}^b +Y_a{}^b, \quad W_a{}^b =\Psi_a{}^b +K_a{}^b,
\end{equation*}
where 
\begin{equation*}
Y_a{}^b =\frac{1}{r}\bigr( \d_a{}^b r_c+{\d_c}^b r_a \bigr)\theta^c,\quad  
K_a{}^b =\big( \d_a {}^b g_{c\conj{d}}+{\d_{c}}^b g_{a\conj{d}}\bigr) \theta^{c}\wedge\theta^{\conj{d}}.
\end{equation*}
We can express $\theta_a{}^b$ and  $W_a{}^b$ in terms of  the Graham--Lee connection $\varphi_\a{}^\b$ of $r$ as  follows:

\begin{prop}[\cite{Marugame}]\label{GL_BE_prop}
Let $\mu=(1+\kappa r)^{-1}$.
With respect to an admissible coframe $\{\theta^{\a}, \theta^{\xi}=\pa r\}$, 
we have
\begin{align*}
{\theta_{\a}}^{\b}&={\varphi_{\a}}^{\b}+i\kappa\vartheta{\d_{\a}}^{\b}, &
{\theta_{\xi}}^{\b}&=-\kappa\theta^{\b}+i{A^{\b}}_{\ol{\g}}\theta^{\ol{\g}}-\kappa^{\b}\ol{\pa}r, \\
{\theta_{\a}}^{\xi}&=\ell_{\a\ol{\g}}\theta^{\ol{\g}}+r\mu\kappa_{\a}\pa r+{ir\mu}A_{\a\g}\theta^{\g}, &
{\theta_{\xi}}^{\xi}&=\kappa\ol{\pa}r-r \kappa^{2}\mu\pa r +{r\mu}\pa\kappa.
\end{align*} 
\end{prop}
Let $\Th_a{}^b=d\theta_a{}^b-\theta_a{}^{c}\wedge\theta_{c}{}^b$.
Then one can also show that $\tensor{\Th}{_a^a}=\tensor{W}{_a^a}$ near $M$ (see~\cite{Marugame}).
Since $r$ is an approximate solution to the Monge--Amp\`{e}re equation, 
\begin{equation}\label{AEE}
\begin{aligned}
\Th_a{ }^a &=W_a{ }^a =\Psi_a{ }^a +K_a{ }^a  \\
 &=\Ric (g)+4g=-\pa\ol{\pa}\log \calJ[r]=\pa\ol{\pa} O(r^{4}). 
\end{aligned}
\end{equation}
On the other hand, by Proposition \ref{GL_BE_prop}, we obtain
\begin{align*}
d{\theta_{\g}}^{\g}&=\Omega_\g{}^\g+2id\kappa\wedge\vartheta +2i\kappa d\vartheta ,\\
d{\theta_{\xi}}^{\xi}&=d\kappa\wedge\ol{\pa}r+\kappa\pa\ol{\pa}r+{\kappa^2}\mu\pa r\wedge\ol{\pa}r -2\kappa r\mu d\kappa\wedge\pa r \\
&\quad+{\kappa^2 r}{\mu^{2}}d(\kappa r)\wedge\pa r+{r\kappa^2\mu}\pa\ol{\pa}r+\mu dr\wedge\pa\kappa \\
&\quad-r{\mu^2}d(\kappa r)\wedge\pa\kappa-{r}\mu\pa\ol{\pa}\kappa,
\end{align*}
where $\tensor{\Omega}{_\alpha^\beta}$ is the curvature of the Graham--Lee connection.
By \cite{GL},
\begin{align*}
\Omega_\g{}^\g&=\Ric_{\g\ol{\mu}}\theta^{\g}\wedge\theta^{\ol{\mu}}-i{A_{\g\b,}}^{\b}\theta^{\g}\wedge\ol{\pa}r -i{A_{\ol{\g}\ol{\b},}}^{\ol{\b}}\theta^{\ol{\g}}\wedge\pa r \\
&\quad-2dr\wedge\bigl(\kappa_{\g}\theta^{\g}-\kappa_{\ol{\g}}\theta^{\ol{\g}}\bigr)-\frac{1}{2}\bigl(\Delta_{b}\kappa+2|A|^2\bigr)\pa r\wedge\ol{\pa}r.
\end{align*}
Thus comparing the coefficients of $\pa r\wedge\ol{\pa}r$ and $\theta^{\g}\wedge\theta^{\ol{\mu}}$
in \eqref{AEE}, we have
\begin{multline}\label{aee1}
2\kappa_{N}-{r}\mu\kappa_{\xi\ol{\xi}}-\frac{1}{2}(\Delta_{b}\kappa +2|A|^2)-2\kappa^2-2\kappa^3 r\mu-{\kappa^2 r^2}{\mu^2}\kappa_{\ol{\xi}} \\
-{\kappa^3 r}{\mu^2}+{r^2}{\mu^2}\kappa_{\xi}\kappa_{\ol{\xi}}+{\kappa r}{\mu^2}\kappa_{\xi}+4\kappa r\mu\kappa_{N}-{r}{\mu}\kappa_{\g}\kappa^{\g}=O(r^2),
\end{multline}
\begin{equation}\label{aee2}
\Ric_{\g\ol{\mu}}-
(3\kappa-r\mu\kappa_{\xi}
+r \kappa^2\mu) \ell_{\g\ol{\mu}}
+r^2\mu^2\kappa_{\g}\kappa_{\ol{\mu}}
-r\mu\kappa_{\g\ol{\mu}}=O(r^{4}).
\end{equation}
From the contraction of \eqref{aee2}, we have  $\kappa^{(0)}=\frac{1}{6}\Scal$. 
Then setting $r=0$ in \eqref{aee1} gives
$$
\kappa^{(1)}=\frac{1}{24}\Delta_{b}\Scal+\frac{1}{36}\Scal^2+\frac{1}{2}|A|^2.
$$ 
Next, we differentiate \eqref{aee1} in the $N$-direction and set $r=0$ to obtain 
\begin{equation}\label{kappa2}
2\kappa^{(2)}-\kappa_{\xi\ol{\xi}}-\frac{1}{2}(\Delta_b \kappa)_N-(|A|^2)_N-3\kappa^3+\kappa
\kappa_{\xi}-\kappa_{\g}\kappa^{\g}=0.
\end{equation}
 In the calculation of each term in \eqref{kappa2}, we need some commutation relations and Bianchi identities for the Graham--Lee connection:
\begin{prop}
The second covariant derivatives of a function $f$ and a tensor $t_\alpha$ satisfy the following commutation relations:
\begin{align}
f_{TN}-f_{NT}&=i(f_{\g} \kappa^\g-f_{\ol{\g}}\kappa^{\ol{\g}})+\kappa f_{T}, \label{commu1}\\
t_{\a,\ol{\b}N}-t_{\a,N\ol{\b}}&=-\frac{i}{2}t_{\a,\g}{A^{\g}}_{\ol{\b}}-\frac{i}{2}t_{\g} 
 {A^{\g}}_{\ol{\b},\a}+\frac{\kappa}{2}t_{\a,\ol{\b}}-\frac{1}{2}t_{\a}\kappa_{\ol{\b}}-t_{\g}\kappa^{\g}\ell_{\a\ol{\b}}. \label{commu2}
\end{align}
\end{prop}
\begin{proof}
We take a local frame for which the connection forms ${\varphi_{\a}}^{\b}$ vanish at a point, and differentiate both sides of 
$$
df=f_N dr+f_T\theta+f_\g \theta^{\g}+f_{\ol{\g}}\theta^{\ol{\g}}.
$$
Comparing the coefficients of $\theta\wedge dr$, we obtain \eqref{commu1}.
Similarly we differentiate both sides of 
$$
dt_{\a}=t_{\a,N}dr+t_{\a,T}\theta+t_{\a,\g}\theta^\g+t_{\a,\ol{\g}}\theta^{\ol{\g}}+t_\g\varphi_\a{}^\g
$$
and compare the coefficients of $\theta^{\ol{\b}}\wedge dr$ to obtain \eqref{commu2}.
\end{proof}
\begin{prop}
Let $n=2$ and let $\theta$ be a pseudo-Einstein contact form. Then
\begin{align}
A_{\a\b,N}&=\kappa A_{\alpha\beta}-i\kappa_{\a\b}-\frac{i}{2}A_{\a\b,T}, \label{BianchiAN} \\
i A_{\a\b,T}A^{\a\b}&=\frac{1}{8}\Scal\Delta_b\Scal+|\nabla A|^2+R_{\a\ol{\b}\g\ol{\d}}A^{\a\g}A^{\ol{\b}\ol{\d}}-\frac{1}{2}|A|^2\Scal+({\rm div}), \label{BianchiScal}
\end{align}
where $|\nabla A|^2={A_{\a\b,}}^{\g}{A^{\a\b}}_{,\g}$ and ${\rm (div)}$ stands for divergence terms.
\end{prop}
\begin{proof}
We differentiate the structure equation \eqref{GLstr} in a frame where the connection forms vanish at 
a point. Then we obtain the complex conjugate of \eqref{BianchiAN} by comparing the coefficients of $\theta\wedge dr\wedge\theta^{\ol{\g}}$. To prove \eqref{BianchiScal}, we use the following Bianchi identities in \cite{Lee}:
\begin{align}
2i{A_{\a\b,}}^{\b}&=\Scal_{\a}, \label{Bianchi-divA} \\
A_{\a\b,\ol{\g}\d}-A_{\a\d,\ol{\g}\b}&=i\ell_{\b\ol{\g}}A_{\a\d,T}-i\ell_{\d\ol{\g}}A_{\a\b,T}
+R_{\a\ol{\g}\b\ol{\mu}}{A^{\ol{\mu}}}_{\d}-R_{\a\ol{\g}\d\ol{\mu}}{A^{\ol{\mu}}}_{\b}.
\end{align}
From these formulas we have 
\begin{align*}
-\Scal\Delta_b\Scal-8|\nabla A|^2&=8({A_{\a\g,}}^{\g}{A^{\a\b}}_{,\b}-{A_{\a\b,}}^{\g}{A^{\a\b}}_{,\g})+({\rm div}) \\
&=8({{A_{\a\b,}}^{\g}}_{\g}-{{A_{\a\g,}}^{\g}}_{\b})A^{\a\b}+({\rm div}) \\
&=8(iA_{\a\b,T}-2iA_{\a\b,T}+{{{R_{\a}}^{\g}}}_{\b\ol{\d}}{A^{\ol{\d}}}_{\g}-{\rm Ric}_{\a\ol{\d}}{A^{\ol{\d}}}_{\b})A^{\a\b} \\
&\quad+({\rm div}) \\
&=-8iA_{\a\b,T}A^{\a\b}+8R_{\a\ol{\b}\g\ol{\d}}A^{\a\g}A^{\ol{\b}\ol{\d}}-4|A|^2\Scal+({\rm div}).
\end{align*}
Thus we obtain \eqref{BianchiScal}.
\end{proof}
Noting that $\xi=N+(i/2)T$ and using the commutation relation \eqref{commu1},
we compute $\kappa_{\xi\ol{\xi}}$ as
\begin{align*}
\kappa_{\xi\ol{\xi}}&=\kappa^{(2)}+\frac{i}{2}(\kappa_{TN}-\kappa_{NT})+\frac{1}{4}\kappa_{TT}
                                 =\kappa^{(2)}+\frac{i}{4}(\kappa^2)_T+\frac{1}{4}\kappa_{TT}
                                 =\kappa^{(2)}+({\rm div}).
\end{align*} 
By using the commutation relation \eqref{commu2}, we compute as
\begin{align*}
{{\kappa_\g}^{\g}}_N&={\kappa_{\g N}}^{\g}-\frac{i}{2}(\kappa_{\g,\mu}A^{\mu\g}+\kappa_\g {A^{\g\mu}}_{,\mu})+\frac{1}{2}\kappa{\kappa_\g}^{\g}-\frac{1}{2}\kappa_{\g}\kappa^{\g}-2\kappa_{\g}\kappa^{\g} \\
&=3\kappa{\kappa_\g}^{\g}+({\rm div}).
\end{align*}
Hence we have
\begin{align*}
(\Delta_b \kappa)_N&={{\kappa_\g}^{\g}}_N+{{\kappa_{\ol{\g}}}^{\ol{\g}}}_N
                          =3\kappa\Delta_b \kappa+({\rm div})
                          =\frac{1}{12}\Scal\Delta_b \Scal+({\rm div}).
\end{align*}
By \eqref{BianchiAN}, we have
\begin{align*}
(|A|^2)_N&=A_{\a\b,N}A^{\a\b}+A_{\ol{\a}\ol{\b},N}A^{\ol{\a}\ol{\b}} \\
&=2\kappa|A|^2-i(\kappa_{\a\b}A^{\a\b}-\kappa_{\ol{\a}\ol{\b}}A^{\ol{\a}\ol{\b}}) 
-\frac{i}{2}(A_{\a\b,T}A^{\a\b}-A_{\ol{\a}\ol{\b},T}A^{\ol{\a}\ol{\b}})\\
&=2\kappa|A|^2-2\kappa{\rm Im}{A_{\a\b,}}^{\a\b}-iA_{\a\b,T}A^{\a\b}+({\rm div}) \\
&=2\kappa|A|^2+\frac{1}{12}\Scal\Delta_b\Scal-iA_{\a\b,T}A^{\a\b}+({\rm div}).
\end{align*}
In the last equality, we have also used
$\im\tensor{A}{_\alpha_\beta_,^\alpha^\beta}=-(1/4)\Delta_b\Scal$, which follows from \eqref{Bianchi-divA}.
Finally, the last two terms in \eqref{kappa2} are computed as
\begin{align*}
	\kappa\kappa_\xi&=\kappa\kappa_N+({\rm div})
	=\frac{1}{144}\Scal\Delta_b\Scal+\frac{1}{12}|A|^2\Scal+\frac{1}{216}\Scal^3+({\rm div}), \\
	\kappa_{\g}\kappa^{\g}&=-\frac{1}{2}\kappa\Delta_b\kappa+({\rm div})
	=-\frac{1}{72}\Scal\Delta_b\Scal+({\rm div}).
\end{align*}
Substituting the results into \eqref{kappa2}, we have
$$
\kappa^{(2)}=\frac{1}{108}\Scal^3+\frac{1}{4}|A|^2\Scal+\frac{5}{48}\Scal\Delta_b\Scal-
iA_{\a\b,T}A^{\a\b}+({\rm div}).
$$
By \eqref{s3}, we obtain
\begin{equation*}
s^{(3)}=-\frac{1}{36}\Scal^3+\frac{3}{2}|A|^2\Scal-\frac{1}{8}\Scal\Delta_b\Scal+
3iA_{\a\b,T}A^{\a\b}+({\rm div}).
\end{equation*}
Using \eqref{BianchiScal}, we also have 
\begin{equation*}
s^{(3)}=-\frac{1}{36}\Scal^3+\frac{1}{4}\Scal\Delta_b\Scal+3|\nabla A|^2
+3R_{\a\ol{\b}\g\ol{\d}}A^{\a\g}A^{\ol{\b}\ol{\d}}+({\rm div}).
\end{equation*}
Comparing this formula with \eqref{Pi5}, we obtain
\begin{equation}\label{old-formula}
\begin{aligned}
\ol{Q}^{\prime}(M)+(4\pi)^3\mu(M)\\=-\int_M \Bigl(
\frac{1}{3}|S|^2\Scal & +4|\nabla A|^{2}-\frac{2}{3}|\partial_b\Scal|^2
\Bigr)\theta\wedge(d\theta)^2.
\end{aligned}
\end{equation}
Since 
$S_{\alpha\ol\beta\gamma\ol\delta}=R_{\alpha\ol\beta\gamma\ol\delta}-
\frac16 \Scal(h_{\alpha\ol\beta}h_{\gamma\ol\delta}+h_{\alpha\ol\delta}h_{\gamma\ol\beta}),
$
the Bianchi identity
$$
R_{\alpha\ol\beta[\rho|\ol\sigma,|\gamma]}=
iA_{\alpha[\gamma|,\ol\beta|}h_{\rho]\ol\sigma}+iA_{\alpha[\gamma|,\ol\sigma|}h_{\rho]\ol\beta}
$$
gives
$$
(\operatorname {div} S)_{\alpha\ol\beta\gamma}=-2iA_{\alpha\gamma,\ol\beta}+\frac{2}{3}\Scal_{(\alpha}h_{\gamma)\ol\beta},
$$
and thus
$$
|\operatorname {div} S|^2=4|\nabla A|^{2}-\frac{2}{3}|\partial_b\Scal|^2.
$$
Therefore \eqref{old-formula} can be also written as \eqref{Q-prime5}.

\setcounter{lem}{0}
\setcounter{equation}{0}
\renewcommand{\thesection}{A}
\setcounter{subsection}{0}

\section*{Appendix by A.\ Rod Gover and Kengo Hirachi
\\
Variation of the obstruction function 
and 
CR invariant differential operators }
 
In this appendix, we compute the first variation of the obstruction
function $\calO$ under deformations.  As is described in \S4, the
deformation of strictly pseudoconvex domains in a complex manifold can
be naturally parametrized by the densities $\calE(1)$ and the
obstruction function takes values in $\calE(-n-2)$.  Thus the first
observation here is that the first variation of the obstruction function
gives a CR invariant differential operator
$$
P_{n+3}\colon\calE(1)\to\calE(-n-2).
$$ 
Below we will express $P_{n+3}$ in terms of the connection of the ambient
metric. This shows that the principal part of $P_{n+3}$ agrees with
that of the power of the sublaplacian $\Delta_b^{n+3}$; see Theorem
\ref{variationOthm}.  The computation is a refinement of the ones in
\cite{GJMS, GG} in the sense that we keep track of the Ricci tensor
and the ambiguity of the ambient metric. In particular, we first show that
the power of ambient Laplacian $\wt\Delta^{n+3}$ induces an operator
$\calE(1)\to\calE(-n-2)$ that depends on the ambiguity in the ambient
metric; see \eqref{PPrelation}. Then we  show in
\S\ref{ambiguity-sec} that we can further normalize the ambient metric
so that $\wt\Delta^{n+3}$ defines a CR invariant operator.

 These results show that a CR invariant differential operator on
 $\calE(1)$, with principal part $\Delta_b^{n+3}$, is not obtained by
 the classical GJMS construction in conformal geometry \cite{GJMS},
 which was applied to the CR case in \cite{GG} via Fefferman's
 conformal structure on a manifold of dimension $2n+2$.  Moreover, in
 the conformal setting, related nonexistence theorems have been
 obtained in \cite{Gr1, GoH}: on a general conformal manifold of even
 dimensions $2N$, there is no conformally invariant linear
 differential operators with principal part $\Delta^k$ for $k>N$.  In
 the CR setting, the corresponding dimension is $n=N-1$.  This is not
 a contradiction since in the CR case the ambient metric has less
 ambiguity; the ambiguity is parametrized by $\calE(-n-2)$, in
 contrast to the general conformal case where it is parametrized by a
 weighted symmetric two tensor (on the original conformal manifold).
 See \S\ref{ambiguity-sec} for more discussion on the ambiguity.

\subsection{The GJMS construction}
We start by recalling the construction of invariant operators in \cite{GJMS, GG}.
For a strictly pseudoconvex hypersurface $M$ in a complex manifold $X$, we first fix 
 an ambient metric $\wt g$ and define differential operators by using the Laplacian of $\wt g$.
We follow the notation in \S\ref{Preliminary}.

\begin{lem}\label{lem-GJMS}
Let $k\in\bZ$ with $k\ge-n/2$.  

\noindent
$(1)$ For $F\in\wt\calE(k)$, 
\begin{equation}\label{GJMS-def}
(\wt\Delta^{n+2k+1}F)|_\calN\in\calE(-n-k-1)
\end{equation}
depends only on $f=F\,|_\calN\in\calE(k)$ and defines 
a  differential operator
$$
\wt\Delta^{n+2k+1}_\calN\colon\calE(k)\to\calE(-n-k-1)
$$
whose principal part agrees with that of the power of the sublaplacian $\Delta_b^{n+2k+1}$.

\noindent
$(2)$ Each $f\in\calE(k)$ can be extended to $F\in\wt\calE(k)$ so that
\begin{equation}\label{GJMS-psi}
\wt\Delta F=\psi\bfr^{n+2k}
\end{equation}
holds for a $\psi\in\wt\calE(-n-k-1)$.  Such an extension is unique modulo $O(\bfr^{n+2k+1})$. Moreover,
$\psi|_\calN$ depends only on $f$ and satisfies 
$$
c_{n,k}\psi|_\calN= \wt\Delta^{n+2k+1}_\calN f, \qquad c_{n,k}=(-1)^n((n+2k)!)^{2}.
$$
\end{lem}

The proof is exactly same as the one for the conformal case
\cite{GJMS}.  The key tool is the commutator of the Laplacian
$\wt\Delta=\wt\nabla_A\wt\nabla^A$ with powers of $\bfr$, acting on
functions on $\wt X$:
\begin{equation}\label{key-commutator}
[\wt\Delta,\bfr^l]=l\,\bfr^{l-1}(\bfr^A\pa_A+\bfr_A\pa^A+n+l+1).
\end{equation}
In particular, for $H\in\wt\calE(w)$, we have
\begin{equation}\label{delta-r}
\wt\Delta(\bfr^{l}{H})
=l(n+l+2w+1)\bfr^{l-1}{H}+\bfr^{l}\wt\Delta{H}.
\end{equation}
Using this equation, we can give an extension $F$ of $f\in\calE(k)$ satisfying $\wt\Delta F=O(\bfr)$  as follows:
Take an arbitrary extension $\wf\in\wt\calE(k)$ and set $F=\wf+\bfr\varphi$ for 
$\varphi\in\wt\calE(k-1)$.  Then
$$
\wt\Delta(\wf+ \bfr \varphi)=\wt\Delta\wf+(n+2k)\varphi+O(\bfr).
$$
Hence, if $n+2k\ne0$, then 
\begin{equation}\label{harmonic-ext}
F=\wf-\frac{1}{n+2k}\bfr\wt\Delta\wf+O(\bfr^2)
\end{equation}
satisfies $\wt\Delta F=O(\bfr)$.  This result
is extended by an obvious finite induction to yield (\ref{GJMS-psi}) and 
then further similar analysis results in the lemma.

More generally, in place of the powers of Laplacian, we can apply
covariant derivatives of type $(n+2k+1,n+2k+1)$ to $F\in\wt\calE(k)$
and then take a complete contraction.  
We consider the following example in the case $k=1$,
$$
\wt\Delta^{n+1}\wt\nabla^{AB}{}_{AB}F,
$$ 
which will appear in the variation of the obstruction function.
Here $\wt\Delta^{n+1}$ is composed with a contraction of holomorphic
covariant derivatives applied twice followed by anti-holomorphic derivatives
applied twice.

\begin{lem}\label{commutator-prop} For $F\in\wt\calE(1)$,
we have
\begin{equation}\label{operator-difference}
\wt\Delta^{n+1}\wt\nabla^{AB}{}_{AB}\,F-\wt\Delta^{n+3}F=
c_{n,1}\big(\calO\wt\Delta F+\calO_A F^A+\wt\Delta\calO\cdot F\,\big)+O(\bfr).
\end{equation}
In particular, if $F=O(\bfr)$, then 
$
\wt\Delta^{n+1}\wt\nabla^{AB}{}_{AB}\,F=O(\bfr)
$
holds.  
\end{lem}

\begin{proof}
From the contracted Ricci identity, we have
$$
\wt\Delta^2 F-
\wt\nabla^{AB}{}_{AB}F=-\wt\nabla^A(\Ric_A{}^B F_B).
$$
Since $\det (\wt g_{A\conj B})=1+
\bfr^{n+2}\calO$, we have $\log\det (\wt g_{A\conj B})=
\bfr^{n+2}\calO+O(\bfr^{2(n+2)})$. Thus the Ricci tensor satisfies
$$
\Ric_{A\conj B}=-\wt\nabla_{A\conj B}(\bfr^{n+2}\calO)
+O(\bfr^{2n+2}).
$$
Thus we have
\begin{align*}
-\Ric_{A}{}^B F_B
=&(\bfr^{n+2}\calO)_{A}{}^B F_B+O(\bfr^{2n+1})
\\
=&
(n+2)(n+1)\bfr^n\bfr_A\calO F\\
&
+(n+2)\bfr^{n+1}(\calO_A F+\calO F_A
+\bfr_A\calO^B F_B)\\
&+\bfr^{n+2}\calO_A{}^B F_B+O(\bfr^{2n+2})
\end{align*}
and so
$$
-\wn^A(\Ric_{A}{}^B F_B)=
(n+2)\bfr^{n+1}(F\,\wt\Delta\calO+\calO_A F^A+\calO\wt\Delta F)+O(\bfr^{n+2}).
$$
Applying $\wt\Delta^{n+1}$ to both sides and then simplifying the right-hand side by using 
\eqref{delta-r},  we obtain  \eqref{operator-difference}.

To prove the second statement, it suffices to show that
$\calO\wt\Delta F+\calO_A F^A$
is $O(\bfr)$ if $F$ is also. If we write $F=\bfr\varphi$ with $\varphi\in\wt\calE(0)$, then
\begin{align*}
\calO\wt\Delta F+\calO_A F^A&
=\calO\wt\Delta(\bfr\varphi)+\calO_A \bfr^A\varphi+O(\bfr)\\
&=(n+2)\calO\varphi-(n+2)\calO\varphi+O(\bfr)\\
&=O(\bfr).\qedhere
\end{align*}
\end{proof}

The second statement of Lemma \ref{commutator-prop} implies that 
$$
\wt\Delta^{n+1}\wt\nabla^{AB}{}_{AB}F\big|_\calN
$$
depends only on $f=F|_\calN$ and hence we can  define a differential operator
$$
P_{n+3}\colon\calE(1)\to\calE(-n-3)
$$  
by $P_{n+3}f=\re \wt\Delta^{n+1}\wt\nabla^{AB}{}_{AB}\,F\big|_\calN$.
It follows from \eqref{operator-difference} that
 the principal part of $P_{n+3}$ agrees with that of $\wt\Delta^{n+3}$.
 
\subsection{Variation of the obstruction function}
Now we are ready to compute the variation of the obstruction function. Consider a family of Fefferman's defining functions $\{\bfr_t\}_t$ and, as in \S\ref{Hamilton-sec},  take a flow $\Phi_t$ generated by the vector field $Y_t=-\re\dot \bfr_t^A\pa_A$,
where $\dot \bfr_t=(d/dt)\bfr_t\in\wE(1)$.  To simplify the notation, we set 
$\bfr=\bfr_0$,
$\dot\bfr=\dot\bfr_{0}$ and
$\delta_0=(d/dt)|_{t=0}$.
\begin{thm}\label{variationOthm}
The first variation of the obstruction function satisfies
\begin{equation}\label{variationO}
\delta_0(\calO_t\circ\Phi_t)|_\calN=c_{n,1}^{-1}\,
P_{n+3}\dot\bfr.
\end{equation}
\end{thm}

 Note that $\delta_0(\calO_t\circ\Phi_t)$ is intrinsically defined from the family of CR structures and hence $P_{n+3}$ can be factored through $P_{\alpha\beta}\colon\calE(1)\to\calE_{\alpha\beta}(1)$, which gives the first variation of the deformation corresponding to $f=-\dot\bfr|_\calN$; see \S\ref{deformation-complex}.
  In particular, $\delta_0(\calO_t\circ\Phi_t)=O(\bfr)$
 if $\dot\bfr|_\calN\in\ker P_{\alpha\beta}$.  This property can be also verified from the right-hand side: if $f\in\ker P_{\alpha\beta}$, then  Proposition \ref{ext-prop} gives an extension such that $\wn_{AB}F=O(\bfrho^\infty)$ and so
$\wt\Delta^{n+1}\wt\nabla^{AB}{}_{AB}F=O(\bfrho^\infty)$ follows.
 
\begin{proof}[Proof of Theorem \ref{variationOthm}]
Since $\Phi_t$ is generated by the vector field $Y_t=-\re\dot \bfr_t^A\pa_A$, we have
\begin{equation}\label{O-fvar}
\delta_0(\calO_t\circ\Phi_t)=\dot\calO-\re\dot\bfr^A\calO_A.
\end{equation}
Recall from Lemma \ref{cor-dotr1} that $\dot\calO$
satisfies
$$
\wt\Delta\dot\bfr=(n+2)\dot\bfr \calO\bfr^{n+1}+\dot\calO\bfr^{n+2}+O(\bfr^{2n+3}).
$$
Using
$
\wt\Delta(\dot\bfr\calO\bfr^{n+2})=(n+2)\dot\bfr\calO\bfr^{n+1}+
\bfr^{n+2}\wt\Delta(\dot\bfr\calO),
$
we have
$$
\wt\Delta(\dot\bfr-\dot\bfr\calO\bfr^{n+2})=\left(\dot\calO-\wt\Delta(\dot\bfr\calO)\right)\bfr^{n+2}+O(\bfr^{2n+3}).
$$
Thus, by Lemma \ref{lem-GJMS}, we obtain 
$\wDelta^{n+3}\dot\bfr=c_{n,1}(\dot\calO-\wt\Delta(\dot\bfr\calO))+O(\bfr)$, or 
\begin{align*}
\dot\calO
&=c_{n,1}^{-1}\wDelta^{n+3}\dot\bfr+\wt\Delta(\dot\bfr\calO)+O(\bfr).
\end{align*}
Substituting this into \eqref{O-fvar} gives 
$$
\delta_0(\calO_t\circ\Phi_t)=c_{n,1}^{-1}\,
\wt\Delta^{n+3}\dot\bfr+\re\dot\bfr^A\calO_A+
\dot\bfr\wt\Delta\calO+O(\bfr),
$$
which agrees with  $c_{n,1}^{-1}\re \wt\Delta^{n+1}\wt\nabla^{AB}{}_{AB}\,\dot\bfr+O(\bfr)$ 
by \eqref{operator-difference}.
  \end{proof}
  
 \subsection{Dependence on the ambient metric}
We next study the dependence of $P_{n+3}$  on the choice of the ambient metric.
Let $\wt g$ be an ambient metric with potential $\bfr$ and $\wh g$  one with potential 
$$
\wh\bfr=\bfr+\varphi\bfr^{n+3}/(n+3),
\qquad \varphi\in\wE(-n-2).
$$
We denote the covariant derivative,  the Laplacian  and the obstruction function for  the metric $\wh g$ by    $\wh\nabla$,   $\wh\Delta$ and $\wh\calO$, respectively.  

\begin{lem} \label{ambiguity-lem}
We have
\begin{align}\label{PPrelation}
\wh\Delta^{n+3}F&=\wt\Delta^{n+3}F+c_{n,1} F\wt\Delta\varphi+O(\bfr), \quad F\in\wt\calE(1);
\\ \label{calOrelation}
\wh\calO&=\calO+\bfr\wt\Delta\varphi/(n+3)+O(\bfr^2).
\end{align}
\end{lem}
If we substitute above two formulas into \eqref{operator-difference}, we obtain
\begin{equation}\label{PPrelation2}
\wh\Delta^{n+1}\wh\nabla^{AB}{}_{AB}\, F=
\wt\Delta^{n+1}\wt\nabla^{AB}{}_{AB}\, F+O(\bfr),
\end{equation}
which shows that $P_{n+3}$ is a CR invariant differential operator. 
This is consistent with \eqref{variationO}, in which the left-hand side is CR invariant by the construction.

\begin{proof}[Proof of Lemma \ref{ambiguity-lem}]
Differentiating $\wh\bfr=\bfr+\varphi\bfr^{n+3}/(n+3)$, we have
\begin{align*}
\wh g_{A\conj B}=\wt g_{A\conj B}&+(n+2)\bfr^{n+1}\bfr_A\bfr_{\conj B}\varphi\\
&
+\bfr^{n+2}(\bfr_{A}\varphi_{\conj B}+\bfr_{\conj B}\varphi_{A}
+\wt g_{A\conj B}\varphi)
+O(\bfr^{n+3}).
\end{align*}
Thus the inverse matrix satisfies
\begin{align*}
\wh g^{A\conj B}=\wt g^{A\conj B}&-(n+2)\bfr^{n+1}\bfr^A\bfr^{\conj B}\varphi\\
&-\bfr^{n+2}(\bfr^{A}\varphi^{\conj B}+\bfr^{\conj B}\varphi^{A}+\wt g^{A\conj B}\varphi)
+O(\bfr^{n+3}),
\end{align*}
where on the right-hand side the indices are raised by using $\wt g^{A\conj B}$.
Hence, for $F\in\wt\calE(1)$,
\begin{align*}
\wh\Delta F&=\wt\Delta F-(n+2)\bfr^{n+1}F\varphi\\
&\quad\qquad-
\bfr^{n+2}\big(
\varphi_AF^A+\varphi^AF_A+\varphi\wt\Delta F\big)+O(\bfr^{n+3})\\
&=\wt\Delta F-(n+2)\bfr^{n+1}F\varphi\\
&\quad\qquad-
\bfr^{n+2}\big(
\wt\Delta(F\varphi)-F\wt\Delta\varphi\big)+O(\bfr^{n+3}).
\end{align*}
We now use Lemma \ref{lem-GJMS}. 
If $\wDelta F=\psi\bfr^{n+2}$ holds, then
\begin{align*}
\wh\Delta F=&-(n+2)\bfr^{n+1}F\varphi\\
&+
\bfr^{n+2}\big(\psi-\wDelta(F\varphi)+F\wt\Delta\varphi\big)+O(\bfr^{n+3}).
\end{align*}
On the other hand, for $F\varphi\in\wt\calE(-n-1)$,
\begin{align*}
\wh\Delta (\bfr^{n+2}F\varphi)&=\wt\Delta  (\bfr^{n+2}F\varphi)+O(\bfr^{n+3})\\
&=(n+2)  \bfr^{n+1}F\varphi+\bfr^{n+2}\wt\Delta (F\varphi)+O(\bfr^{n+3}).
\end{align*}
Thus, setting $\wh F=F+\bfr^{n+2}F\varphi$, we have
$\wh\Delta\wh F=\wh\psi\,\bfr^{n+2}+O(\bfr^{n+3})$, where
$$
\wh\psi=\psi+F\wt\Delta\varphi.
$$
Since $c_{n,1}\wh\psi|_\calN=\wh\Delta^{n+3}F|_\calN$, we get the desired formula.

To prove the second formula, we take the variation of the determinant in the homogenous coordinates $X^A$ (see \S \ref{sec-ambient-metric}):
\begin{align*}
(-1)^{n+1}\det\wh g_{A\conj B}
&=(-1)^{n+1}(\det\wt g_{A\conj B})
(1+\wt\Delta(\bfr^{n+3}\varphi)/(n+3)+O(\bfr^{n+4}))\\
&=(-1)^{n+1}\det\wt g_{A\conj B}+\bfr^{n+3}\wt\Delta\varphi/(n+3)+O(\bfr^{n+4}).
\end{align*}
Thus the obstruction function of $\wh\bfr$ satisfies
$
\wh\calO=\calO+\bfr\wt\Delta\varphi/(n+3)+O(\bfr^2).
$
\end{proof}

\subsection{Normalizing the ambient metric}\label{ambiguity-sec}
  In view of  \eqref{PPrelation}, we see that $\wt\Delta^{n+3}$ defines a CR invariant operator if we can further normalize the ambiguity $\varphi$ so that $\wt\Delta\varphi=O(\bfr)$ holds. Such a normalization is obtained by imposing 
\begin{equation}\label{strictF}
\wt\Delta\calO=O(\bfr),
\end{equation}
where $\wt\Delta$ is defined with respect to the ambient metric with potential $\bfr$.
We  call a Fefferman's defining function satisfying \eqref{strictF}  a {\em strict Fefferman's defining function}.

\begin{prop}
For any strictly pseudoconvex $M\subset X$, 
there exist  strict Fefferman's defining functions of $\calN=K_M^*\subset\wt X$.
If $\bfr$ and $\wh\bfr$ are two such defining functions for $\calN$, then
$$
\wh\bfr-\bfr=\bfr^{n+3}\varphi
$$
 for  $\varphi\in\wt\calE(-n-2)$ such that
$\wt\Delta\varphi=O(\bfr)$.
\end{prop}

\begin{proof}
Let $\wh g$ be the ambient metric with potential
$\wh\bfr=\bfr+\bfr^{n+3}\varphi/(n+3)$.
Applying $\wt\Delta$ to \eqref{calOrelation} gives
$
\wt\Delta\wh\calO=
\wt\Delta\calO-(n+4)/(n+3)\wt\Delta\varphi+O(\bfr).
$
Since $\wh\Delta\wh\calO=\wt\Delta\wh\calO+O(\bfr^{n+1})$, we obtain
\begin{equation}\label{DeltaO}
\wh\Delta\wh\calO=
\wt\Delta\calO-\frac{n+4}{n+3}\wt\Delta\varphi+O(\bfr).
\end{equation}
Thus setting $\varphi=-(n+3)/(n+4)^2\bfr\wt\Delta \calO$, 
we obtain $\wh\Delta\wh\calO=O(\bfr)$.

If $\bfr$ and $\wh\bfr$ are strict Fefferman's defining functions, then
\eqref{DeltaO} forces $\wt\Delta\varphi=O(\bfr)$.
\end{proof}

With this refinement of the ambient metric, we can generalize the GJMS construction of CR invariant differential operators.

\begin{thm}
 Let $\wt g$ be the ambient metric defined from a strict Fefferman's defining function.  Then $
\wt\Delta^{n+3}\colon\wt\calE(1)\to\wt\calE(-n-2)$ induces a CR invariant differential operator
$\wt\Delta^{n+3}_\calN\colon\calE(1)\to\calE(-n-2)$.
\end{thm}

\begin{rem}\rm A further normalization of defining function has been done in 
\cite{H01}, where an exact formal solution to the Monge--Amp\`ere equation with logarithmic singularity was used. 
\end{rem}

\section*{Acknowledgments}
We would like to thank Jeffrey Case for pointing out an error in the computation of the total $Q$-prime curvature in dimension 5 in the earlier version of this paper.
 
K.H was partially supported by JSPS KAKENHI 60218790
and 15H02057.
T.M was partially supported by JSPS Research Fellowship for Young Scientists and KAKENHI 13J06630.
Y.M was partially supported by JSPS Research Fellowship for Young Scientists and KAKENHI 14J11754.
A.R.G gratefully acknowledges support from the Royal Society of New Zealand via Marsden Grant 13-UOA-018.


\begin{thebibliography}{99}
\bibitem{AGL} 
T. Akahori, P. M. Garfield,  J. M. Lee,
{\em Deformation theory of 5-dimensional CR structures and the Rumin complex},
Michigan Math. J.
{\bf 50} (2002), 517--550.

\bibitem{Al} 
P. Albin, 
{\em Renormalizing curvature integrals on Poincar\'e--Einstein manifolds}, Adv. in Math. {\bf 221} (2009), 140--169.
\bibitem{An} 
M. T. Anderson, {\em $L^2$ curvature and renormalization of AHE metrics on 4-manifolds}, Math. Res. Lett. {\bf 8}, (2001), 171--188.

\bibitem{BE1} D. Burns, C. L. Epstein, {\em A global invariant for three dimensional CR-manifolds}, Invent. Math. {\bf 92} (1988), 333--348.

\bibitem{BE2} D. Burns, C. L. Epstein, {\em Characteristic numbers of bounded domains}, Acta Math. {\bf 164} (1990) 29--71.

\bibitem{BFM} 
T. P. Branson, L. Fontana, C. Morpurgo, {\em Moser-Trudinger and Beckner-Onofri's inequal-ities on the CR sphere}, Ann. of Math. \textbf{177} (2013),  1--52.

\bibitem{CL} J.-H. Cheng, J.\ M.\ Lee, {\em The Burns-Epstein invariant and deformation of CR structures}, Duke Math. J. {\bf 60} (1990), 221--254.

\bibitem{CSS} A. \v{C}ap, J. Slov\'{a}k, V. Sou\v{c}ek, {\em Bernstein-Gelfand-Gelfand Sequences}, 
Ann. of Math. {\bf 154} (2001), 97--113.


\bibitem{CGY} J. S. Case, A. R. Gover, {\em
The $P^\prime\!$-operator, the $Q^\prime\!$-curvature, and the CR tractor calculus}, in praparation.

\bibitem{CaY} J. S. Case, P. Yang,
{\em A Paneitz-type operator for CR pluriharmonic functions},
{ Bull. Inst. Math. Acad. Sin.} {\bf 8} (2013), 285--322.

\bibitem{ChengYau} 
S.-Y. Cheng, S.-T. Yau,
{\em On the existence of a complete K\"ahler metric on non-compact complex manifolds and the regularity of Fefferman's equation},
Comm. Pure  Appl. Math. {\bf 33} (1980), 507--544. 

\bibitem{CaoChang} J. Cao, S.-C. Chang, {\em Pseudo-Einstein and Q-flat Metrics with Eigenvalue Estimates on CR-hypersurfaces}, Indiana Univ. Math. J. {\bf 56} (2007), 2839-2857.

\bibitem{DT}
S. Dragomir, G. Tomassini,
Differential Geometry and Analysis on CR Manifolds,
Progress in Math. {\bf 246}, Birkh\"auser, 2006.


  \bibitem{Farris} F. A. Farris,  
{\em An intrinsic construction of Fefferman's CR metric},
Pacific J. Math. {\bf 123} (1986), 33--45.
  
  \bibitem{F1} C.\ L. Fefferman,  
{\em Monge--Amp\`ere equations, the Bergman kernel, and geometry of pseudoconvex domains}, Ann. of Math. {\bf 103} (1976), 395--416; erratum {\bf 104} (1976), 393--394.
  
 \bibitem{F2} C. L. Fefferman, {\em Parabolic invariant theory in complex analysis}, Adv. in  Math. {\bf 31} (1979), 131--262.

\bibitem{FG2} C. L. Fefferman, C. R. Graham, 
$Q$-curvature and Poincar\'e metrics,
{\em Math. Res. Lett.} {\bf 10} (2003), 819--832.

\bibitem{FG3} C. L. Fefferman, C. R. Graham, 
{\em The ambient metric}, Princeton Univ. Press. 2011.

\bibitem{FH} C. L.~Fefferman, K.~Hirachi,
{\em Ambient metric construction of $Q$-curvature in conformal and CR geometries,} Math. Res. Lett. {\bf 10} (2003), 819--832.

\bibitem{GG} A. R. Gover, C.R. Graham, {\em CR invariant powers of the sub-Laplacian}, J. Reine Angew. Math. {\bf 583} (2005), 1--27.

\bibitem{GoH} A. R. Gover, K. Hirachi, 
 Conformally invariant powers of the Laplacian --- A complete non-existence theorem, Jour. Amer. Math. Soc. {\bf 17} (2004), 389--405.
  
\bibitem{Gr2} C. R. Graham, {\em Higher asymptotics of the complex Monge--Amp\`ere equation}, Compositio Math. {\bf 64} (1987), 133--155.

\bibitem{Gr1}
C. R. Graham,
{\em Conformally invariant powers 
of the Laplacian, II: Nonexistence}, J. London Math. Soc. {\bf 46} (1992), 566--576.

\bibitem{Gr3} C. R. Graham, {\em Volume and area renormalizations for conformally compact Einstein metrics}, The Proceedings of the 19th Winter School ``Geometry and Physics'' (Srn\'\i, 1999), Rend. Circ. Mat. Palermo (2) Suppl. {\bf 63} (2000), 31--42.

\bibitem{GrH} C. R. Graham, K. Hirachi,
{\em The ambient obstruction tensor and Q-curvature. AdS/CFT correspondence}, Einstein metrics and their conformal boundaries, 59--71, IRMA Lect. Math. Theor. Phys., {\bf 8}, Eur. Math. Soc., Z\"urich, 2005.

\bibitem{GrH2} C. R. Graham, K. Hirachi,
{\em Inhomogeneous Ambient Metrics,} The IMA volumes in mathematics and its applications 
{\bf 144}, 403--420, Springer 2007.

\bibitem{GJMS}
C. R. Graham, R. Jenne, L. J. Mason,  G. A. J. Sparling, 
{\em Conformally invariant powers 
of the Laplacian, I: Existence}, J. London Math. Soc. {\bf 46} (1992) 557--565.

\bibitem{GL} C. R. Graham, J. M. Lee, {\em Smooth solutions of degenerate Laplacians on strictly pseudoconvex domains}, Duke Math. J. {\bf 57} (1988) 697--720.

\bibitem{GMS}
C. Guillarmou, S. Moroianu, J.-M. Schlenker, 
{\em The renormalized volume and uniformisation of conformal structures},
Journal Institut Math. Jussieu, to appear. {\tt arXiv:1211.6705}

\bibitem{H1} K. Hirachi, 
{\em  Scalar pseudo-hermitian invariants and the Szeg\"o kernel on 
three-dimensional CR manifolds}, {in ``Complex Geometry,"}
Lect. Notes in Pure and Appl.\ Math.\ {\bf 143}, pp 67--76, Dekker, 1992.

\bibitem{H01} K. Hirachi, {\em Construction of boundary invariants and the logarithmic singularity of the Bergman kernel}, Ann. of Math. {\bf 151} (2000), 151--190.

\bibitem{H2} K. Hirachi, 
{\em Q-prime curvature on CR manifolds},
Diff. Geom. Appl. {\bf 33} Suppl. (2014), 213--245.

\bibitem{H3} K. Hirachi, 
{\em $Q$ and $Q$-prime curvature in CR geometry}, Proceedings of the ICM, Seoul 2014, vol. III,  257--277.

\bibitem{HMO}
K. Hirachi, Y. Matsumoto,  B. {\O}rsted, in preparation.

\bibitem{HPT} P. Hislop, P. Perry, A.-H. Tang,
{\em CR-invariants and the scattering operator for complex manifolds with boundary}, Analysis \& PDE {\bf 1} (2008), 197--227.

\bibitem{Lee}  J. M. Lee, {\em Pseudo-Einstein structures on CR manifolds}, Amer.\ J.\ Math. {\bf 110} (1988) 157--178.

\bibitem{LM} 
J. M. Lee, R. Melrose, {\em Boundary behaviour of the complex Monge-Amp\`ere equation}, Acta Math. {\bf 148} (1982) 159--192.

\bibitem{Marugame} T. Marugame,  
{\em Renormalized Chern-Gauss-Bonnet formula for complete K\"ahler--Einstein metrics},
Amer. J. Math {\bf 138} (2016), 1067--1094.

\bibitem{Matsumoto} Y. Matsumoto,  
{\em GJMS operators, $Q$-curvature, and obstruction tensor of partially integrable CR manifolds},
Diff. Geom. Appl. {\bf 45} (2016), 78--114.

\bibitem{S} N. Seshadri,
{\em Volume renormalization for complete Einstein--K\"ahler metrics}, Differential
Geom. Appl. {\bf 25} (2007), 356--379.

\bibitem{T} N. Tanaka,  A differential geometric study on strongly pseudo-convex manifolds, Kinokuniya Book-Store Co. Ltd., Tokyo, 1975.

\bibitem{CQY} 
P. Yang, S.-Yu. A. Chang, J. Qing,
{\em
On the renormalized volumes for conformally compact Einstein manifolds},
Jour. Math. Sci. {\bf 149} (2008), 1755--1769.
\end{thebibliography}
\end{document}